\documentclass{amsart}

\usepackage{amssymb}
\usepackage[cp1250]{inputenc}
\usepackage{picture}
\usepackage{color}
\usepackage{tikz}
\usepackage{tikz}
\usetikzlibrary{matrix}

\newtheorem{theorem}{Theorem}[section]
\newtheorem{lemma}[theorem]{Lemma}
\newtheorem{example}[theorem]{Example}
\newtheorem{proposition}[theorem]{Proposition}
\newtheorem{problem}[theorem]{Problem}
\newtheorem{corollary}[theorem]{Corollary}

\newtheorem{remark}[theorem]{Remark}

\newcommand{\N}{\mathbb N}

\newcommand{\R}{\mathbb R}

\newcommand{\ve}{\varepsilon}

\author{Jacek Marchwicki, Piotr Nowakowski and Franciszek Prus-Wi\'{s}niowski}
\address{Faculty of Mathematics and Computer Science
\\University of Warmia and Mazury in Olsztyn
\\S{\l}oneczna 54,
10-710 Olsztyn
\\Poland
\\ORCID: 0000-0003-3712-8289}
\email{jacek.marchwicki@uwm.edu.pl}

\address{Faculty of Mathematics and Computer Science
\\University of \L\'{o}d\'{z}
\\Banacha 22,
90-238 \L\'{o}d\'{z}
\\Poland
\\ORCID: 0000-0002-3655-4991}
\email{piotr.nowakowski@wmii.uni.lodz.pl}

\address{Instytut Matematyki\\
Uniwersytet Szczeci\'{n}ski\\
ul. Wielkopolska 15\\
PL-70-453 Szczecin\\
Poland
\\ ORCID 0000-0002-0275-6122}
\email{franciszek.prus-wisniowski@usz.edu.pl}

\thanks{This research was funded in part by National Science Centre, Poland, Grant number: 2022/06/X/ST1/00764.
}
\title{Algebraic sums of achievable sets involving Cantorvals}

\makeatletter
\@namedef{subjclassname@2020}{%
 \textup{2020} Mathematics Subject Classification}
\makeatother


\begin{document}

\begin{abstract}
In this paper we look at the topological type of algebraic sum of achievement sets. We show that there is a Cantorval such that the algebraic sum of its $k$ copies is still a Cantorval for any $k \in \N$. We also prove that for any $m,p \in (\N\setminus \{1\}) \cup \{\infty\}$, $p \geq m$, the algebraic sum of $k$ copies of a Cantor set can transit from a Cantor set to a Cantorval for $k=m$ and then to an interval for $k=p$. These two main results are based on a new characterization of sequences whose achievement sets are Cantorvals. We also define a new family of achievable Cantorvals which are not generated by multigeometric series. In the final section we discuss various decompositions of sequences related to the topological typology of achievement sets.

\end{abstract}

\maketitle
\section{Introduction}

Achievement sets, that is, the sets of all possible subsums of absolutely convergent series, have been investigated for over a century now. The historically first paper devoted to the achievement sets (\cite{Kakeya}) aimed at finding all possible topological types of such sets. The problem was solved only in 1988 by Guthrie and Nymann in \cite{GN88}. After the publication of their fundamental theorem, the investigation of achievement sets gained momentum. Among others, achievement sets served as a counterexample (see \cite{Sannami92} and \cite{PWT18}) to the Palis hypothesis \cite{Palis87} that the arithmetic sum (or difference) of two Cantor sets, both with Lebesgue measure zero, is either of Lebesgue measure zero or it contains an interval. Palis hypothesis came from the theory of dynamical systems and added much interest to the investigation of algebraic sums of Cantor sets which is a thriving field of research (see, for example, \cite{PY97}, \cite{E07}, \cite{Pourba18}, \cite{T19}). Achievement sets, due to their rigid nature, are relatively small family of compact sets in the real line as can be seen from the fact that Palis hypothesis is generically true for dynamically defined Cantor sets \cite{MY01}. The elementary nature of achievement sets is often misleading because many seemingly straightforward facts require quite complicated proofs or constructions. In our opinion, the most challenging open problem related to achievement sets is to find an easy to use characterization of the absolutely convergent series that generate Cantor sets.

Let $\sum a_n$ be an absolutely convergent series of real terms. 
\textsl{The achievement set} of the sequence $(a_n)_{n=1}^\infty$ is defined by
$$
A\ =\ A(a_n)\ :=\ \left\{x\in\mathbb R:\qquad \exists I\subset\mathbb N \quad x\,=\,\sum_{n\in I} a_n\,\right\}.
$$
The set of indices $I$ is called a \textsl{representation} of $x$. It does not have to be unique, although we will not dwell on it (unlike the papers \cite{cardfun} and \cite{GM23}). As it was known already to Kakeya, an achievement set is compact and if the sequence $(a_n)$ has infinitely many non-zero terms, then $A(a_n)$ must be perfect. Since we will be focused on the topological character of achievement set, we assume throughout the paper that the terms of the series $\sum a_n$ are nonnegative and nonincreasing, since $A(a_n)$ is merely a translation of $(A(|a_n|)$ by the number $\sum_{a_n<0}a_n$ and hence $A(a_n)$ and $A(|a_n|)$ are homeomorphic always. A set $S\subset\mathbb R$ is said to be \textsl{achievable} if there is a sequence $(a_n)$ such that $S=A(a_n)$ \cite[p. 519]{Jones}. Given any positive integer $k$, we define \textsl{$k$-initial subsums} as
$$
F_k\ =\ F_k(a_n)\ :=\ \left\{x\in\mathbb R: \qquad \exists I\subset\{1,2,\ldots,k\,\}\quad x\,=\,\sum_{n\in I}a_n\,\right\}.
$$
We will be using the increasing arrangement $F_k=(f_j^k)_{j=1}^{m_k}$. Always $k+1\le m_k\le 2^k$. More generally, if $S$ is any finite set  (or multiset) of real numbers, the set of all numbers that are sums of at least one subset (also possibly a multiset) of $S$, will be denoted by $\Sigma S$. In particular, $F_k=\Sigma\{a_1,a_2,\ldots,a_k\,\}$. For all $k\in\mathbb N$ the equality $A(a_n)\,=\, F_k+E_k$ holds where $E_k$ denotes the achievement set of the remainder sequence $(a_n)_{n=k+1}^\infty$.

Next, we define the \textsl{$k$-th iterate} of $A$ by
$$
I_k\ :=\ \bigcup_{f\in F_k}[f,\,f+r_k].
$$
Each $I_k$ is a \textsl{multi-interval set}, that is, the union of a finite family of closed and bounded intervals. In the classic case of $a_n=\tfrac2{3^n}$, when $A$ is the ternary Cantor set $C$, the set $I_k$ is exactly the set obtained in the $k$-th step of the standard geometric construction of $C$. Additionally, we will write $I_0:=[0,\,\sum a_n]$. 
Always $A(a_n)\,=\,\bigcap_nI_n$. 

We denote the family of all connectivity components of $I_{k-1}\setminus I_k$ by $\mathcal{G}_k$. Open intervals belonging to $\mathcal{G}_k$ are called \textsl{$A$-gaps of order $k$}. It is not difficult to see that the family $\mathcal{G}_k$ is nonempty if and only if $r_k<a_k$. Intervals of the form $(r_n,\,a_n)$ are gaps whenever $r_n<a_n$ (The First Gap Lemma, see \cite{BFPW1}) and will be called the \textsl{principal gaps} of $A$. We will say that an $A$-gap $G$ is \textsl{dominating} if all $A$-gaps lying to the left of $G$ are shorter than $G$. Nontrivial components of $A$ will be called \textsl{$A$-intervals}. It was already proven by Kakeya that $A(a_n)$ is a multi-interval set if and only if $a_n\le r_n$ for all sufficiently large $n$. $A(a_n)$ is a single interval if and only if $a_n\le r_n$ for all indices $n$. We say then that $(a_n)$ is an \textsl{interval-filling} sequence and this terminology comes not from Kakeya, but from \cite{DJK}. Kakeya observed also that if $a_n>r_n$ for all sufficiently large $n$ then $A(a_n)$ is a \textsl{Cantor set}, that is, a set in $\mathbb R$ homeomorphic to the classic Cantor ternary set or, equivalently, a nonempty bounded, perfect and nowhere dense set (see \cite[Thm. 3.3]{Foran}). He conjectured in \cite{Kakeya} that if the series $\sum a_n$ has infinitely many positive terms, then $A(a_n)$ is either a multi-interval set or a Cantor set which turned out to be false, but only after seventy years. Series (or sequences) satisfying $a_n>r_n$ for all $n$ are called \textsl{fast convergent}.

A set $P\subset \mathbb R$ is said to be a Cantorval if it is homeomorphic to the set
$$
GN\ :=\ C\,\cup\,\bigcup_n G_{2n-1}
$$
where $C$ is the Cantor ternary set $C=A(\tfrac2{3^n})$ and $G_{2n-1}$ is the union of all $4^{n-1}$ $C$-gaps of order $2n-1$. It is known that a Cantorval is exactly a nonempty compact set in $\mathbb R$ such that it is the closure of its interior and both endpoints of every nontrivial component are accumulation points of its trivial components. Other topological characterizations of Cantorvals can be found in \cite{MO} and \cite{BFPW1}.

The fundamental Guthrie-Nymann classification theorem (\cite[Thm.1]{GN88}, \cite{NS0}) asserts that the achievement set of an absolutely convergent series always is of one of the following four topological types: a finite set, a multi-interval set, a Cantor set or a Cantorval. Proving that a particular series generates a Cantorval is rather difficult and thus almost all known examples of such series are the \textsl{multigeometric series} whose general term is of the form $a_{(n-1)k+i}\,=\,l_iq^n$, $n\in\mathbb N$, $i\in\{1,2,\ldots,k\}$ for some $k\in\mathbb N$, $q\in(0,\,1)$ and some real numbers $l_1\ge l_2\ge \ldots\ge l_k>0$. Such a multigeometric sequence will be denoted by $(l_1,l_2,\ldots,l_k;\,q)$ and the notation will be used in the Thm. \ref{sufficientcantorval}. The only two known examples of families of non-multigeometric Cantorvals can be found in \cite{VMPS19} and \cite{FN23}.

One of known and easy to use sufficient conditions for the achievement set to be a Cantor set uses the notion of a semi-fast convergent series \cite{BFPW2}. A series $\sum a_n$ with monotonic and positive terms convergent to 0 is called \textsl{semi-fast convergent} if it satisfies the condition
$$
a_n\ >\ \sum_{k:\,a_k<a_n}a_k \qquad\text{for all $n$}.
$$
If $\sum a_n$ is a semi-fast convergent series, then there exist two uniquely determined sequences, $(\alpha_k)$ of positive numbers decreasing to 0 and $(N_k)$ of positive integers, such that 
$$
a_n\ =\ \alpha_k\qquad\qquad \text{for} \qquad \sum_{j=0}^{k-1}N_j\ <\ n\ \le\ \sum_{j=0}^kN_j
$$
where $N_0:=0$. The numbers $\alpha_k$ are the values of the terms of the series $\sum a_n$ and $N_k$ is the multiplicity of the value $\alpha_k$ in the series $\sum a_n$. Thus, we can identify
$$
\sum a_n\ =\ \sum(\alpha_k,\,N_k) \qquad\qquad\text{or}\qquad\qquad (a_n)\ =\ (\alpha_k,\,N_k)
$$ 
and the sum of the series is $\sum a_n\,=\,\sum\,(\alpha_k,\,N_k)\,=\,\sum_{k=1}^\infty\alpha_kN_k$. Theorem 16 of \cite{BFPW2} states that if $\sum a_n$ is semi-fast convergent then $A(a_n)$ is a Cantor set.

Let $(b_n)$ and $(c_n)$ be two sequences of real numbers. We will say that a sequence $(a_n)$ is the \textsl{union of the sequences} $(b_n)$ and $(c_n)$ if there is a partition $\mathbb N=N\sqcup M$ with both subsets $N$ and $M$ infinite such that
\begin{equation}
\label{unionseq}
(b_n)_{n\in\mathbb N}\ =\ (a_k)_{k\in M} \qquad \text{and} \qquad (c_n)_{n\in\mathbb N}\ =\ (a_k)_{k\in N}.
\end{equation}
We will write $(a_n) = (b_n) \cup (c_n)$.
Loosely speaking, it means that $(a_n)$ is the mixture of all terms of $(b_n)$ and all terms of $(c_n)$. We will also say that the sequences $(b_n)$ and $(c_n)$ form a decomposition of the sequence $(a_n)$. The admissible partition $\mathbb N=M\sqcup N$ is not unique if the sequences
$(b_n)$ and $(c_n)$ have a common value. If $(a_n)$ is the union of sequences $(b_n)$ and $(c_n)$, and $\mathbb N=M\sqcup N$ is a fixed partition satisfying \eqref{unionseq}, then we will write $a_k\in(b_n)$ whenever $k\in M$. It is not difficult to see that if $\sum b_n$ and $\sum c_n$ are convergent series of positive and nonincreasing terms, then $A(b_n)+A(c_n)=A(a_n)$.

In the end of this section we present additional notation used in this paper. Let $B \subset \R$. By $\lambda(B)$ we denote the Lebesgue measure of a set $B$. If $B$ is finite, then by $|B|$ we denote the cardinality of the set $B$. If $B$ is an interval, then by $|B|$ we denote the length of $B$.

\section{General remarks on achievable sets}

We will start with a rather general observation on gaps of achievable sets. But first, let us recall The Second Gap Lemma.

\begin{lemma}[The Second Gap Lemma, \cite{BFPW1}]
Let $(\alpha,\,\beta)$ be an $A$-gap of order $k$. Then $\beta\in F_k$ and hence $\beta=f_j^k$ for a unique $j\in\{2,\,3,\,\ldots,\,m_k\,\}$. Moreover, $\alpha=f_{j-1}^k+r_k$.
\end{lemma}

\begin{proposition}[The $k$-th order gap lemma]
\label{prop1}
The principal gap of order $k$ has the maximal length among all $A$-gaps of order $k$.
\end{proposition}
\begin{proof}
Let $(\alpha,\,\beta)$ be an $A$-gap of order $k$. From The Second Gap Lemma, we have $\beta\in F_k$, and hence $\beta=f_j^k$ for a unique $j\in\{2,\,3,\,\ldots,\,m_k\,\}$. It follows that $\alpha=f_{j-1}^k+r_k$.

If $f_j^k$ has a representation $B$ with $k\in B$, then $f_{j-1}^k\ge f_j^k-a_k$ and hence,
$$
\beta-\alpha\ =\ f_j^k-(f_{j-1}^k+r_k)\ \le a_k-r_k,
$$
that is, the gap $(\alpha,\,\beta)$ is no longer than the principal gap $(r_k,\,a_k)$.

If $f_j^k$ has no representation involving $k$, then $k\ge2$ and $f_j^k=f_i^{k-1}$ for some $i\in\{2,\,3,\,\ldots,\,m_{k-1}\}$. We need to consider two cases.

\textsl{Case} 1: $f_{j-1}^k$ has no representation with $k$. Since there are gaps of order $k$, we have $r_k < a_k$, and so
$$
A\ \supset\ F_k\ \ni\ f_{j-1}^k+a_k\ >\ f_{j-1}^k+r_k\ =\ \alpha,
$$
and hence $f_{j-1}^k+a_k>\beta$ which implies that
$$
\beta\,-\,\alpha\ \le\ (f_{j-1}^k+a_k)\,-\,(f_{j-1}^k+r_k)\ =\ a_k\,-\,r_k.
$$

\textsl{Case} 2: $f_{j-1}^k$ has a representation with $k$. Then $f_{j-1}^k=f_s^{k-1}+a_k$ for some $s\in\{1,\,2,\,\ldots,\,m_{k-1}\}$. We have $f_s^{k-1} < f_{j-1}^k < f_j^k =f_{i}^{k-1}$, thus $f_{i-1}^{k-1}\ge f_s^{k-1}$. Now, if $f_{i-1}^{k-1}=f_s^{k-1}$, then $$\alpha=f_{j-1}^k+r_k=f_{i-1}^{k-1}+a_k+r_{k}=f_{i-1}^{k-1}+r_{k-1}$$ and $\beta=f_i^{k-1}$ which means that $(\alpha,\,\beta)$ is a gap of order at most $k-1$, a contradiction. Therefore, we have $f_{i-1}^{k-1}>f_s^{k-1}$, but then $F_k\ni f_{i-1}^{k-1}+a_k > f_s^{k-1}+a_k=f_{j-1}^k$ which implies that $f_{i-1}^{k-1}+a_k\ge f_j^k=\beta$. Since $f_j^k=f_i^{k-1}$, it must be $f_{j-1}^k\ge f_{i-1}^{k-1}$ and hence $f_{i-1}^{k-1}+r_k\le f_{j-1}^k+r_k=\alpha$. Finally,
$$
\beta\,-\,\alpha\ \le (f_{i-1}^{k-1}+a_k)\,-\,(f_{i-1}^{k-1}+r_k)\ =\ a_k\,-\,r_k
$$ and the proof is completed.
\end{proof}

As a simple corollary we obtain the powerful and frequently used well-known result \cite[Lemma 2.4]{recover}.

\begin{corollary}[The Third Gap Lemma]
Every dominating gap is principal.
\end{corollary}

Recall that a set $A$ in a topological space is said to be regularly closed if it is the closure of its interior, that is, $A=\overline{\text{int}\,A}$. The notion provides one more topological characterization of Cantorvals that we formulate below and that, unlike the other two principal topological characterizations of Cantorvals (cf. \cite[pp.331 and 343]{MO} and \cite[Thm. 21.17]{BFPW1}), does not evoke neither $A$-gaps nor $A$-intervals.

\begin{theorem}
\label{charCvl}
A bounded set $A\subset\mathbb R$ is a Cantorval if and only if it is regularly closed and its boundary $\text{Fr}\,A$ is a Cantor set.
\end{theorem}

Before proving it, let us note that, thanks to the above equivalence, the topological classification of achievement sets of absolutely convergent series amounts to choosing one property of the boundary: $\text{Fr}\,A$ is either a finite set or a Cantor set, and choosing the relative size of the interior: $A$ is either nowhere dense or regularly closed. Any of four such combinations of these properties characterizes exactly one topological type of achievement sets.

\begin{corollary}
\label{cortyp}
Let $A$ be an achievement set of an absolutely convergent series. Then $A$ is closed and bounded and:
\begin{itemize}
\item[(i)] $A$ is nowhere dense and $\text{Fr}\,A$ is finite if and only if $A$ is a finite set;
\item[(ii)] $A$ is nowhere dense and $\text{Fr}\,A$ is a Cantor set if and only if $A$ is a Cantor set;
\item[(iii)] $A$ is regularly closed and $\text{Fr}\,A$ is finite if and only if $A$ is a multi-interval set;
\item[(iv)] $A$ is regularly closed and $\text{Fr}\,A$ is a Cantor set if and only if $A$ is a Cantorval.
\end{itemize}
\end{corollary}
Before proving Theorem \ref{charCvl} let us recall another characterization of Cantorvals.
\begin{theorem}{\cite[Theorem 21.17]{BFPW1}} \label{21.17}
A nonempty perfect and compact set $P \subset \R$ is a Cantorval if and only if $P$-gaps and $P$-intervals do not have common points and the union of all $P$-intervals is dense in $P$.  

\end{theorem}
Now, we can prove Theorem \ref{charCvl}.
\begin{proof} [Proof of \emph{Theorem \ref{charCvl}}] \

Let $ A \subset \mathbb R$ be bounded.

($\Rightarrow$) Suppose that $A$ is a Cantorval. Then $A$ is closed and the union of all $A$-intervals is dense in $A$ (Theorem \ref{21.17}). Hence, the union of interiors of all $A$-intervals is dense in $A$ as well, but the last union is the interior of $A$. Thus, $A$ is regularly closed. The boundary of $A$ is a bounded nowhere dense closed subset of $\mathbb R$. It is nonempty, because it contains all endpoints of $A$-intervals. It remains to show that $\text{Fr}\,A$ is a perfect set. Indeed, if there was an isolated point in $\text{Fr}\,A$, it would be a common endpoint of an $A$-interval and an $A$-gap, because $A$ has no isolated points. Such common endpoints do not exist by Theorem \ref{21.17}, and hence $\text{Fr}\,A$ has no isolated points.

($\Leftarrow$) Suppose that $A$ is regularly closed and its boundary is a Cantor set. Then the interior of $A$ equals to the union of interiors of all $A$-intervals which is contained in the union of all $A$-intervals, which in turn is a subset of $A$. Passing to closures, we obtain
$$
A\ =\ \overline{\text{int}\,A}\ \subset\ \overline{\text{the union of all $A$-intervals}}\ \subset \ A,
$$
and thus the union of all $A$-intervals is dense in $A$. Since $\text{Fr}\,A$ is a Cantor set, it has no isolated points, and hence $A$-gaps and $A$-intervals cannot have common endpoints. From Theorem \ref{21.17} it follows that $A$ is a Cantorval.
\end{proof}

Note that the starting assumption of Theorem \ref{charCvl} cannot be removed as shows the example of any Cantorval with its external gaps attached.

Each multi-interval set $W$ (that is, the union of a finite family of closed and bounded intervals) is the union of infinitely many different finite families of closed and bounded intervals. However, one of those families stands out and is uniquely determined by $W$, namely, the family of all connectivity components of $W$. Thus, unless specified otherwise, by writing $W=\bigcup_{i=1}^nP_i$, we mean that $\{P_i: \ 1\le i\le n\,\}$ is the family of connectivity components of $W$. Given a multi-interval set $W=\bigcup_{i=1}^nP_i$, we define
$$
||W||\ :=\ \max_{1\le i\le n}|P_i|.
$$
If $(W_n)_{n\in\mathbb N}$ is a descending sequence of multi-interval sets, then $(||W_n||)_{n\in\mathbb N}$ is a nonincreasing sequence bounded from below and hence it converges to a nonnegative number.

\begin{lemma}
\label{multlem}
Let $(W_n)_{n\in\mathbb N}$ be a descending sequence of multi-interval sets. Then $\bigcap_nW_n$ contains and interval if and only if $\lim\limits_{n\to\infty}||W_n||>0$.
\end{lemma}
\begin{proof}
($\Rightarrow$) If $[a,\,b]\subset \bigcap_nW_n$, then $||W_n||\ge b-a$ for all $n$.

($\Leftarrow$) Suppose that $g:=\lim\limits_{n\to\infty}||W_n||>0$. Let $P_n=[a_n,\,b_n]$ be the most left of all components of $W_n$ of maximal length. Then $|P_n|=||W_n||\ge g$. Let $s_n$ be the middle point of the interval $P_n$. Clearly, $s_n\in W_n\subset W_1$ and hence the sequence $(s_n)_{n\in\mathbb N}$ is bounded. There is a convergent subsequence $(s_{n_k})_{k\in\mathbb N}$. In particular, there exists $M\in\mathbb N$ such that $|s_{n_k}-s_{n_l}|<g/2$ for all $l>k>M$. Then $s_{n_k}\in P_{n_l}$ and thus $s_{n_k}\in P_{n_k}\cap P_{n_l}$. Now, since $W_{n_l}\subset W_{n_k}$ and $P_{n_k}\cap P_{n_l}$ is nonempty, it must be $P_{n_l}\subset P_{n_k}$. The sequence $(P_{n_i})_{i\in\mathbb N}$ is descending and $|P_{n_i}|\ge ||W_{n_i}||\ge g>0$ for all $i$. Therefore, $\bigcap_iP_{n_i}$ is a non-degenerated closed interval contained in $\bigcap_nW_n$.
\end{proof}

We are now going to give an analytic characterization of those achievement sets that contain an interval. Unfortunately, the limit whose value is decisive is too unwieldy to compute in most cases and makes the characterization unsatisfactory except for some very special cases like, for example, the Ferens series \cite{BP} or the Guthrie-Nymann-Jones series \cite{recover}, \cite{B}.

Given an $\epsilon>0$, a finite increasing sequence $(f_i)_{i=m}^n$ will be called \textsl{$\epsilon$-close} either if $m=n$ or if $m<n$ and the distance between any two consecutive terms of the sequence does not exceed $\epsilon$. An $\epsilon$-close subsequence of an increasing sequence $F=(f_i)_{i=1}^k$ will be called \textsl{maximal} if it is not contained in any longer $\epsilon$-close subsequence of $F$. The set of all maximal $\epsilon$-close subsequences of a sequence $F$ will be denoted by $M_\epsilon[F]$. Every finite increasing sequence of real numbers has a unique decomposition into finitely many disjoint maximal $\epsilon$-close subsequences. Given an $\epsilon$-close sequence $(f_i)_{i=m}^k$, we define the \textsl{stretch} of the sequence by $S[(f_i)]:= f_k-f_m$. Finally, given a finite increasing sequence $F$, we define
$$
\Delta_\epsilon F\ :=\ \max\bigl\{S[(f_i)]:\ (f_i)\in M_\epsilon[F]\,\bigr\}.
$$
Now, let $\sum a_n$ be a convergent series of nonnegative and nonincreasing terms. For any $n\in\mathbb N$, the set $F_n$ of $n$-initial subsums $F_n=\{\sum_{i\in A}a_i\colon A\subset\mathbb \{1,2, \ldots,n\} \,\}$ forms an increasing and finite sequence. Recall that the symbol $r_n$ denotes the value of the $n$-th reminder of the series $\sum a_i$, that is, $r_n=\sum_{i=n+1}^\infty a_i$.

\begin{proposition}
\label{charint}
Let $\sum a_n$ be a convergent series of nonnegative and nonincreasing terms. Then the following conditions are equivalent:
\begin{itemize}
\item[(i)] the achievement set $A(a_n)$ contains an interval;
\item[(ii)] \ $\lim\limits_{n\to\infty}\Delta_{r_n}F_n\,>\,0$;
\item[(iii)] \ $\lim\limits_{k\to\infty}\Delta_{r_{n_k}}F_{n_k}\,>\,0$ for some increasing sequence $(n_k)$ of indices.
\end{itemize}
\end{proposition}
\begin{proof}
It is well-known (see \cite{BFPW1}) that
$$ A(a_i)\ =\ \bigcap_{n=1}^\infty I_n,
$$
where each $I_n=\bigcup_{f\in F_n}[f,\,f+r_n]$ is a multi-interval set. Since $||I_n||=\Delta_{r_n}F_n\,+\, r_n$, we have \ $\lim ||I_n||\,=\,\lim \,\Delta_{r_n}F_n$ and hence by the Lemma \ref{multlem} $A(a_i)$ contains an interval if and only if $\lim\limits_{n\to\infty}\Delta_{r_n}F_n>0$.

The equivalence (ii)\ $\Leftrightarrow$\ (iii) and, actually, the equality of the limits follows from the fact that the sequence $(r_n+\Delta_{r_n}F_n)_{n\in\mathbb N}$ is nonincreasing.
\end{proof}

\vspace{.2in}

\section{A family of achievable Cantorvals}

We are going to present a new family of series whose achievement sets are Cantorvals. Unlike the most of the known and broadly used examples of achievable Cantorvals (\cite{WS}, \cite{F}, \cite{GN88}, \cite{Jones}, \cite{recover}, \cite{BP}, \cite{B}, \cite{MM}), the series in our family do not need to be multigeometric. Our new family of achievable Cantorvals is large enough to serve later in this paper as a useful tool in finding a much needed example that cannot be found among Cantorvals generated by multigeometric series (see the Theorem \ref{immortal}).

Given a sequence $m=(m_n)$ of positive integers greater than or equal to 2, a sequence $k=(k_n)$ of positive integers such that $k_n> m_n$ for all $n$, and a sequence $q=(q_n)$ of positive numbers, we will be saying that a series $\sum a_n$ is a \textsl{generalized Ferens series} (shortly, a GF series) if for $i\in\mathbb N$, $n\in\big\{K_{i-1}+1,\,K_{i-1}+2,\,\ldots,\,K_i\,\big\}$, where $K_j:=\sum_{i=1}^jk_i$ for $j\in\mathbb N$ and $K_0:=0$, we have
$$
a_n\ =\ a_n(k,m,q)\ :=\ (m_i+K_i-n)q_i.
$$
Given positive integers $p,r$ with $r>p\ge2$, we will use the symbol $s(p,r):=\sum_{i=1}^{r-1}(p+i)$. We will also write $s_n:=s(m_n,k_n)$. Given an $n\in\mathbb N$, the set of all possible sums formed by integrands taken without repetition from the set
$$
\left\{a_p: \ p\in\{K_{n-1}+1,\, K_{n-1}+2,\,\ldots,\,K_n\,\}\,\right\}
$$
is exactly the set
$$
\bigl(\{0\}\cup\{m_n,\,m_n+1,\,\ldots,\,s_n\}\cup\{s_n+m_n\}\bigr) q_n \qquad \text{(see \cite[Fact 3]{BP})}.
$$

The next theorem gives the sufficient condition for the achievement set of a generalized Ferens series to be a Cantorval. A geometrical idea for this condition is similar to the one used in the paper \cite{Now2}.
\begin{theorem}
\label{nmg1}
If $\sum a_n(m,k,q)$ is a convergent GF series satisfying the conditions
\begin{equation*}
\label{gf1}
\tag{GF$_1$} \ \forall_{n\in\mathbb N}\qquad q_n\ \le\ (s_{n+1}-m_{n+1}+1)q_{n+1}
\end{equation*}
and
\begin{equation*}
\label{gf2}
\tag{GF$_2$} \ \forall_{n\in\mathbb N} \qquad m_nq_n\ >\ \sum_{i>n}(s_i+m_i)q_i,
\end{equation*}
then $A(a_n)$ is a Cantorval.
\end{theorem}

\begin{proof}
Since $(a_n)$ is clearly decreasing for all $n\in \{K_{i-1},  \ldots, K_i\}$, $i \in \N$ and \eqref{gf2} implies that $a_{K_i} = m_iq_i>(m_{i+1}+k_{i+1}-1)q_{i+1}=a_{K_i+1}$, we obtain that $(a_n)$ is decreasing. Our proof will rest on an application of the Proposition \ref{charint} and hence we want to find long sequences in the sets $F_n$ with all consecutive terms sufficiently close. For that purpose, the following additional symbol will be quite useful:
$$
C_n\ :=\ \{m_nq_n,\,(m_n+1)q_n,\,\ldots,\,s_nq_n\,\} = \{a_{K_n+1}, \ldots, a_{K_{n+1}}\} \qquad \text{for $n\in\mathbb N$.}
$$
We are going to define inductively a special sequence $(D_n)_{n\in\mathbb N}$. First, we set $D_1:=C_1$. Clearly, $D_1\subset F_{K_1}$ and the distance of any two consecutive elements of $D_1$ equals to $q_1$. Moreover, $\min D_1=m_1q_1$ and $\max D_1=s_1q_1$.

Suppose now that a set $D_n$ has been defined and satisfies the following conditions
\begin{itemize}
\item[($\alpha_n$)] \ $D_n\,\subset \, F_{K_n}$;
\item[($\beta_n$)] \ the distance between any two consecutive points of $D_n$ does not exceed $q_n$;
\item[($\gamma_n$)] \ $\min D_n\,=\,\sum_{i=1}^nm_iq_i$ and $\max D_n\,=\,\sum_{i=1}^ns_iq_i$.
\end{itemize}
Then define $D_{n+1}:=D_n+C_{n+1}$. We see immediately that the set $D_{n+1}$ satisfies the conditions ($\alpha_{n+1}$) and ($\gamma_{n+1}$). In order to demonstrate ($\beta_{n+1}$), consider any two consecutive elements $f<g$ of $D_{n+1}$. Then $g=h_1+pq_{n+1}$ for some $h_1\in D_n$ and some $p\in\{m_{n+1},\ldots,\,s_{m+1}\}$. If $p>m_{n+1}$, then $h_1+(p-1)q_{n+1}\in D_{n+1}$ and it must be $h_1+(p-1)q_{n+1}\le f$, because $f<g$ are two consecutive elements of $D_{n+1}$. Hence $g-f\le g-h_1-(p-1)q_{n+1}=q_{n+1}$. If $p=m_{n+1}$, then from the fact that $g>f\ge\, \min D_{n+1}$, it follows $h_1>\min D_n$. Hence, by ($\beta_n$), we can find $h_2\in D_n$ such that $h_1-q_n\le h_2<h_1$. Of course,
\begin{equation}
\label{doda}
h_2\,+\,m_{n+1}q_{n+1}\ <\ h_1\,+\,m_{n+1}q_{n+1}\ =\ g.
\end{equation}
On the other hand,
\begin{align*}
h_2\,+\,s_{n+1}q_{n+1}\ &\ge h_1\,-\,q_n\,+\,s_{n+1}q_{n+1}\\
&=\ h_1+\,m_{n+1}q_{n+1}\,+\,(s_{n+1}-m_{n+1}+1)q_{n+1}\,-\,q_n\,-\,q_{n+1}\ \overset{\eqref{gf1}}{\ge}\ g\,-\,q_{n+1}.
\end{align*}
Therefore,
$$
h_2\,+\,\min C_{n+1}\ \overset{\eqref{doda}}{<}\ g\ \le\ h_2\,+\,\max C_{n+1}\,+\,q_{n+1}.
$$
Because any two consecutive elements of $h_2+C_{n+1}$ lie in the distance $q_{n+1}$, there is exactly one element of $h_2+C_{n+1}$ belonging to $[g-q_{n+1},\,g)$. This element belongs to $D_{n+1}$ which completes the proof of the property ($\beta_{n+1}$). Thus, by induction, the sets $D_n$ satisfying ($\alpha_n$), ($\beta_n$),($\gamma_n$) exist for every $n\in\mathbb N$. The elements of $D_n$ form an increasing finite sequence that is $q_n$-close and hence
$$
\Delta_{q_n}F_{K_n}\ \ge\ \max D_n\,-\,\min D_n\ =\ \sum_{i=1}^n(s_i-m_i)q_i.
$$
Since
$$
r_{K_n}\,=\,\sum_{i>n}(s_i+m_i)q_i\ >\ (s_{n+1}-m_{n+1}+1)q_{n+1}\ \overset{\eqref{gf1}}{\ge}\ q_n,
$$
it follows
$$
\Delta_{r_{K_n}}F_{K_n}\ \ge\ \sum_{i=1}^n(s_i-m_i)q_i
$$
 and finally
 $$
 \lim\limits_{n\to\infty}\Delta_{r_n}F_n\ =\ \lim\limits_{n\to\infty}\Delta_{r_{K_n}}F_{K_n}\ \ge\ \sum_{i=1}^\infty(s_i-m_i)q_i\ >\ 0.
 $$
Thus, the achievement set $A(a_n)$ contains an interval, by Proposition \ref{charint}. Since
 $$
 r_{K_n}\,=\,\sum_{i>n}(s_i+m_i)q_i\ \overset{\eqref{gf2}}{<}\ m_nq_n\,=\,a_{K_n}
 $$
 for all $n$, the set $A(a_n)$ has infinitely many gaps. Therefore, by the Guthrie-Nymann
Classification Theorem, $A(a_n)$ is a Cantorval.
 \end{proof}
Now, we will show that there exist GF series satisfying the assumptions of Theorem \ref{nmg1}. We will require from the series some additional properties which will be useful later.
 \begin{theorem}
\label{nmg2}
For every sequence $m=(m_n)$ of positive integers greater than 1 and for any sequence $(c_n)$ with $c_n>1$ and $\alpha:=\sup_n\frac{c_n}{m_n}<1$, there are sequences $k=(k_n)\in\mathbb N^\mathbb N$ with $k_n>m_n$ and $q=(q_n)\in(0,\,1)^\mathbb N$ such that the GF series $\sum a_n(m,k,q)$ is convergent, $A(a_n)$ is a Cantorval, and
\begin{equation}
\label{gwiazd}
(s_{n+1}+m_{n+1})q_{n+1}\ <\ c_nq_n \qquad \text{for all $n\in\mathbb N$},
\end{equation}
and
\begin{equation}
\label{2gwiazd}
c_n\ <\ (1-\alpha)(s_n+m_n) \qquad \text{for all $n\in\mathbb N$}.
\end{equation}
\end{theorem}
\begin{proof}
Take $q_1>0$. Choose $k_1>m_1$ such that
\begin{equation*}
c_{1}\ <\ (1-\alpha)\bigl(s(m_{1},k_1)+m_{1}).
\end{equation*}
We can do it because $s(m_{1},k)\nearrow +\infty$ as $k\to\infty$.

Now, suppose that $l$ is a positive integer such that $k_1,\,\ldots,\,k_l$ and $q_1,\,\ldots,\,q_l$ have been chosen and satisfy \eqref{gwiazd} for all $n<l$ and \eqref{2gwiazd} for all $n\le l$. Since $s(m_{l+1},k)\nearrow +\infty$ as $k\to\infty$, there is $T\in\mathbb N$, $T>m_{l+1}$ such that
\begin{equation}
\label{3gwiazd}
\bigl(s(m_{l+1},\,T)-m_{l+1}+1\bigr)\ >\ \frac{2m_{l+1}}{c_l-1}
\end{equation}
and
\begin{equation*}
c_{l+1}\ <\ (1-\alpha)\bigl(s(m_{l+1},\,T)+m_{l+1}).
\end{equation*}
Define $k_{l+1}:=T$ and $q_{l+1}:=\frac{q_l}{s_{l+1}-m_{l+1}+1}$ where $s_{l+1}=s(m_{l+1},k_{l+1})$. Then
$$
c_lq_l=c_lq_{l+1}(s_{l+1}-m_{l+1}+1)\,\overset{\eqref{3gwiazd}}{>}\,(s_{l+1}-m_{l+1}+1)q_{l+1}+2m_{l+1}q_{l+1}>(s_{l+1}+m_{l+1})q_{l+1},
$$
that is, \eqref{gwiazd} holds for $n=l$. Thus, by induction, two sequences $k=(k_n)$ and $q=(q_n)$ have been defined and they satisfy \eqref{gwiazd} and \eqref{2gwiazd}.

We are now going to show that $\sum a_n(m,k,q)$ is a convergent GF series satisfying \eqref{gf1} and \eqref{gf2}. Clearly,
$$
\sum_{n=1}^\infty a_n\ =\ \sum_{n=1}^\infty\sum_{l=K_{n-1}+1}^{K_n} a_l\ \ =\ \ \sum_{n=1}^\infty(s_n+m_n)q_n
$$
and
\begin{equation} \label{1-alfa}
\frac{(s_{n+1}+m_{n+1})q_{n+1}}{(s_n+m_n)q_n}\ \overset{\eqref{gwiazd}}{<}\ \frac{c_n}{s_n+m_n}\ \overset{\eqref{2gwiazd}}{<}\ 1\,-\alpha
\end{equation}
for all $n$ which proves the convergence of $\sum a_n$. Observe that \eqref{gf1} is satisfied, by the definition of the sequence $(q_{n})$. Furthermore,
\begin{align*}
\sum_{i>n}(s_i+m_i)q_i\ \overset{\eqref{1-alfa}}{<}\ \sum_{j=0}^\infty(s_{n+1}+m_{n+1})q_{n+1}(1-\alpha)^j\ =\ \frac1\alpha(s_{n+1}+m_{n+1})q_{n+1}\ \overset{\eqref{gwiazd}}{<}\frac{c_n}{\alpha}q_n\ \le \ m_nq_n.
\end{align*}
Therefore, \eqref{gf2} holds as well and hence, by Theorem \ref{nmg1}, $A(a_n)$ is a Cantorval.
\end{proof}
In all of the known achievable Cantorvals their Lebesgue measure is equal to the measure of their interior. It also holds for Cantorvals satisfying the assumptions of Theorem \ref{nmg1}.
\begin{theorem}
\label{nmg3}
If $\sum a_n(m,k,q)$ is a convergent GF series satisfying the conditions \eqref{gf1} and \eqref{gf2}, then the Lebesgue measure of the Cantorval $A(a_n)$ equals to the measure of its interior.
\end{theorem}
\begin{proof}
If the series $\sum a_n(m,k,q)$ satisfies the assumptions of our theorem, then $A(a_n)$ is a Cantorval by the virtue of Theorem \ref{nmg1}. Observe that $A$-gaps of order $k$ exist if and only if $k=K_n$ for some $n\in\mathbb N$. Moreover, $A=\bigcap_{n\in\N}I_{K_n}$. Each $I_{K_n}$ is a multi-interval set and every component of $I_{K_n}$ divides into $3$ components of $I_{K_{n+1}}$ with distance between them equal to $m_{n+1}q_{n+1}-r_{K_{n+1}}$. Therefore, the set $I_{K_{n-1}}\setminus I_{K_n}$, that is, the union of all $A$-gaps of order $K_n$, consists of $2\cdot3^{n-1}$ open intervals, each of length $m_nq_n-r_{K_n}$. Thus,
\begin{align}
\label{cruz}
\lambda (A)\ =\ &\lambda\left([0,\,\sum_{i=1}^\infty(m_i+s_i)q_i]\setminus\bigl([0,\,\sum_{i=1}^\infty(m_i+s_i)q_i]\setminus A\bigr)\right)\\
&=\ \sum_{i=1}^\infty(m_i+s_i)q_i\ -\ \lambda\left(\bigcup_{n=1}^\infty\bigl(I_{K_{n-1}}\setminus I_{K_n}\bigr)\right) \notag \\
&=\ \sum_{i=1}^\infty(m_i+s_i)q_i\ -\ \sum_{n=1}^\infty2\cdot3^{n-1}\bigl(m_nq_n\ -\ \sum_{i=n+1}^\infty(m_i+s_i)q_i\bigr). \notag
\end{align}
Each iteration $I_{K_n}$ for $n\in\mathbb N_0$ has a central interval of the form $\left[\sum_{i=1}^nm_iq_i,\,\sigma-\sum_{i=1}^nm_iq_i\right]$ (with the convention $\sum_{i=1}^0m_iq_i=0$). The intersection of all these central intervals is the interval $\left[\sum_{i=1}^\infty m_iq_i,\,\sum_{i=1}^\infty s_iq_i\right]$ which is a connectivity component of $A$. The iteration $I_{K_1}$ consists of three intervals, the central of which is involved in producing the central interval of $A(a_n)$. The other two intervals of $I_{K_1}$: $B:=\bigl[0,\,\sum_{i=2}^\infty(m_i+s_i)q_i\bigr]$ and $C:=\bigl[(m_1+s_1)q_1,\,(m_1+s_1)q_1+\sum_{i=2}^\infty(m_i+s_i)q_i\bigr]$ give rise to two $A$-intervals $\bigl[\sum_{i=2}^\infty m_iq_i,\,\sum_{i=2}^\infty s_iq_i\bigr]$ and $\bigl[(m_1+s_1)q_1+\sum_{i=2}^\infty m_iq_i,\,(m_1+s_1)q_1+\sum_{i=2}^\infty s_iq_i\bigr]$ concentric with $B$ and $C$, respectively.

The iteration $I_{K_2}$ consists of nine disjoint closed intervals. Three of them are concentric with component intervals of $I_{K_1}$. Each of the remaining six is concentric with an $A$-interval of length $\sum_{i=3}^\infty(s_i-m_i)q_i$. Continuing in this fashion, we see that for any $n\in\mathbb N_0$, each component interval of $I_{K_n}$ is concentric with an $A$-interval. Let $\mathcal{P}$ be the family of all $A$-intervals concentric with at least one component of at least one iteration $I_{K_n}$, $n\in\mathbb N_0$. For an interval $P\in\mathcal{P}$ let us define
$$
n_P-1\ :=\ \min\bigl\{l:\ I_{k_l} \ \text{has a component concentric with $P$}\,\bigr\}.
$$
Then $|P|=\sum_{i=n_P}^\infty(s_i-m_i)q_i$. Moreover, given an positive integer $n\ge2$, there are exactly $2\cdot3^{n-2}$ $A$-intervals $P$ for which $n_P=n$ and there is exactly one $A$-interval $P$ for which $n_P=1$. Thus,
\begin{align*}
\lambda(\text{int}\,A)\ &\ge\ \sum_{n=1}^\infty\,\sum_{\substack{P\in\text{$A$-intervals}\\n_P=n}}\ =\ \sum_{i=1}^\infty(s_i-m_i)q_i\ +\ \sum_{n=2}^\infty 2\cdot3^{n-2}\sum_{i=n}^\infty(s_i-m_i)q_i\\
&=\ \sum_{i=1}^\infty(s_i+m_i)q_i\ -\ 2\left(\sum_{i=1}^\infty m_iq_i\ -\ \sum_{n=2}^\infty3^{n-2}\sum_{i=n}^\infty(s_i-m_i)q_i\right)\\
&=\ \sum_{i=1}^\infty(s_i+m_i)q_i\ -\ 2\left(\sum_{n=1}^\infty m_nq_n\ +\ \sum_{n=1}^\infty2\cdot3^{n-1}\sum_{i=n+1}^\infty m_iq_i\ -\ \sum_{n=1}^\infty 3^{n-1}\sum_{i=n+1}^\infty(s_i+m_i)q_i\right)\\
&=\ \sum_{i=1}^\infty(s_i+m_i)q_i\ -\ 2\left(\sum_{n=1}^\infty m_nq_n\ +\ \sum_{i=2}^\infty m_iq_i\sum_{n=2}^i2\cdot3^{n-2}\ -\ \sum_{n=1}^\infty 3^{n-1}\sum_{i=n+1}^\infty(s_i+m_i)q_i\right)\\
&=\ \sum_{i=1}^\infty(s_i+m_i)q_i\ -\ 2\left(\sum_{n=1}^\infty m_nq_n\ +\ \sum_{i=2}^\infty (3^{i-1}-1)m_iq_i\ -\ \sum_{n=1}^\infty 3^{n-1}\sum_{i=n+1}^\infty(s_i+m_i)q_i\right)\\
&=\ \sum_{i=1}^\infty(s_i+m_i)q_i\ -\ 2\left(m_1q_1\ +\ \sum_{n=2}^\infty3^{n-1}m_nq_n\ -\ \sum_{n=1}^\infty 3^{n-1}\sum_{i=n+1}^\infty(s_i+m_i)q_i\right)\\
&=\ \sum_{i=1}^\infty(s_i+m_i)q_i\ -\ 2\left(\sum_{n=1}^\infty 3^{n-1}\bigl(m_nq_n\ -\ \sum_{i=n+1}^\infty(s_i+m_i)q_i\bigr)\right)\ \ \overset{\eqref{cruz}}{=} \ \lambda (A).
\end{align*}
\end{proof}

\vspace{.2in}

\section{Algebraic sums of achievable sets}
The main goal of this section is to check what types of sets we can obtain as algebraic sums of various combinations of achievement sets. In particular, we are interested what we can get as a sum of two or more copies of the same Cantorval or a Cantor set.

We use the following result, which comes from \cite{BBFS}.

\begin{theorem}\label{sufficientcantorval}
Let $a_1\geq a_2\geq \ldots\geq a_m$ be positive integers such that $a_i-a_{i+1} \leq a_m$ for $i \in \{1,2,\ldots,m-1\}$. Assume that there exist positive integers $n_0$ and $r$ such that $\Sigma\{a_1,a_2,\ldots,a_m\}\supset\{n_0,n_0+1,\ldots,n_0+r\}$. If $q\geq\frac{1}{r+1}$, then $A(a_1,a_2,\ldots,a_m;q)$ has a nonempty interior. If $q<\frac{a_m}{\sum_{i=1}^{m}a_i+a_m}$, then $A(a_1,a_2,\ldots,a_m;q)$ is not a finite union of intervals. Consequently, if $\frac{1}{r+1}\leq q <\frac{a_m}{\sum_{i=1}^{m}a_i+a_m}$, then $A(a_1,a_2,\ldots,a_m;q)$ is a Cantorval.
\end{theorem}

\begin{proposition}\label{sumoftwocantorvals}
There exists an achievable Cantorval $D$ such that $D+D$ is also a Cantorval.
\end{proposition}
\begin{proof}
Consider $D_q=A(14,13,12,11,10,9,8,7;q)$. We have \\$\Sigma \{7,8,9,10,11,12,13,14\}=\{0,7,8,9,\ldots,76,77,84\}$. By Theorem \ref{sufficientcantorval}, the set $D_q$ is a Cantorval for all $\frac{1}{71}\leq q <\frac{1}{13}$.

We have $D_q+D_q=A(14,14,13,13,12,12,11,11,10,10,9,9,8,8,7,7;q)$ and $$\Sigma \{14,14,13,13,12,12,11,11,10,10,9,9,8,8,7,7\}=\{0,7,8,9,\ldots,160,161,168\}.$$ Using again Theorem \ref{sufficientcantorval}, we get that $D_q+D_q$ is a Cantorval for every $\frac{1}{155}\leq q <\frac{1}{25}$.
We have proved that both sets $D_q$ and $D_q+D_q$ are Cantorvals for each $\frac{1}{71}\leq q <\frac{1}{25}$.
\end{proof}

\begin{proposition}\label{sumoftwocantorvals2}
Let $k\in\mathbb{N}$. There exists an achievable Cantorval $D$ such that $\underbrace{D+D+\ldots+D}_{k \ \text{times}}$ is also a Cantorval.
\end{proposition}
\begin{proof}
Take natural $m > \frac{1+3k+\sqrt{9k^2+42k+1}}{6}$.
Consider $D_q=A(2m,2m-1,\ldots,m+2,m+1,m;q)$ and $B = \{m,m+1, \ldots, 2m\}.$ We have $\Sigma B=\{0,m,m+1,m+2,\ldots,\frac{3m^2+m}{2}-1,\frac{3m^2+m}{2},\frac{3m^2+3m}{2}\}$. By Theorem \ref{sufficientcantorval}, the set $D_q$ is a Cantorval for all $\frac{2}{3m^2-m+2}\leq q <\frac{2}{3m+5}$.

We have $\underbrace{D_q+\ldots+D_q}_{k \ \text{times}}=A(\underbrace{2m,\ldots,2m}_{k \ \text{times}},\underbrace{2m-1,\ldots,2m-1}_{k \ \text{times}},\ldots,\underbrace{m,\ldots,m}_{k \ \text{times}};q)$ and 
$$\Sigma \{\underbrace{m,\ldots,m}_{k \ \text{times}},\ldots,\underbrace{2m,\ldots,2m}_{k \ \text{times}}\} =\{0,m,m+1,m+2,\ldots,\frac{3m^2k+3mk-2m}{2}-1,\frac{3m^2k+3mk-2m}{2},\frac{3m^2k+3mk}{2}\}.$$
From Theorem \ref{sufficientcantorval} we know that $\underbrace{D_q+\ldots+D_q}_{k \ \text{times}}$ is a Cantorval for $\frac{2}{3m^2k+3mk-4m+2}\leq q <\frac{2}{3mk+3k+2}$. Thus, both $D_q$ and $\underbrace{D_q+\ldots+D_q}_{k \ \text{times}}$ are Cantorvals for $\frac{2}{3m^2-m+2}\leq q <\frac{2}{3mk+3k+2}$. Since $m > \frac{1+3k+\sqrt{9k^2+42k+1}}{6}$, we have $3m^2-m+2 > 3mk+3k+2$, therefore there is $q \in [\frac{2}{3m^2-m+2}, \frac{2}{3mk+3k+2}).$
\end{proof}

From \cite[Theorem 6.4.]{GM23} (and its proof) we obtain the following theorem.
\begin{proposition}\label{CantorvalCantorval}
For any achievable Cantorval $D$ there exists an achievable Cantor set $C$ such that $D+C$ is a Cantorval.
\end{proposition}

\begin{proposition}\label{CantorvalInterval}
For any achievable Cantorval $D$ there exists an achievable Cantor set C such that $D+C$ is an interval.
\begin{proof}
Let $D=A(a_n)$ and let $P:=\{n\in\mathbb N: \ (r_n,\,a_n) \text{ is a dominating gap}\,\}$. Arrange all elements of $P$ in the increasing order $P=(p_i)_{i=1}^\infty$. The series $\sum_{i\in\mathbb N}a_{p_i}$ is fast convergent.

We are going to define an increasing sequence $(n_i)_{i=1}^\infty$ of positive integers and another sequence $(N_i)_{i=1}^\infty$ of positive integers, the later one not necessarily monotone. We start with $n_1=N_1 :=1$ and continue by induction for $k\ge2$, defining
\begin{equation}
\label{dos}
n_k\ :=\ \min\left\{ s>n_{k-1}:\ \sum_{i\ge s}(2a_{p_i}-r_{p_i})\ <\ r_{p_{n_{k-1}}}\ \right\}
\end{equation}
and then choosing $N_k$ to be the unique positive integer such that
\begin{equation}
\label{uno}
a_{p_{n_{k-1}}}-r_{p_{n_{k-1}}}\ \le\ N_ka_{p_{n_k}}\ <\ a_{p_{n_{k-1}}}-r_{p_{n_{k-1}}}+a_{p_{n_k}}.
\end{equation}
The series $\sum_i(a_{p_{n_i}}; N_i)$ is semi-fast convergent and thus its set of subsums $C:=A\bigl(a_{p_{n_i}}; N_i\bigr)$ is a Cantor set \cite[Thm.16]{BFPW2}. Indeed,

\begin{align*}
\sum_{i>k}N_ia_{p_{n_i}}\ &\overset{\eqref{uno}}{<}\ \ \sum_{i=k}^\infty(a_{p_{n_i}}-r_{p_{n_i}})\ +\ \sum_{i>k}a_{p_{n_i}}\\[.1in]
&\le\ a_{p_{n_k}}-r_{p_{n_k}}\,+\,\sum_{j\ge n_{k+1}}(a_{p_j}-r_{p_j})\,+\ \sum_{j\ge n_{k+1}}a_{p_j}\ \overset{\eqref{dos}}{<}\ a_{p_{n_k}}-r_{p_{n_k}}\,+\,r_{p_{n_k}}=a_{p_{n_k}}.
\end{align*}

It remains to observe that the unique monotone mix $(c_n)$ of sequences $(a_n)$ and $(a_{p_{n_i}};N_i)$ is slowly convergent. Clearly, $(c_n)$ is also the monotone mix of two other sequences: $(a_n)_{n\not\in\{p_{n_k}:\,k\in\mathbb N\,\}}$ and $(a_{p_{n_i}};N_i+1)$. In particular, $c_n\le r_n^c$ [the last symbol stands for the $n$-th remainder of the series $\sum c_n$ ] for all such $n$ that do not satisfy
\begin{equation}
\label{tres}
\exists_{i\in\mathbb N} \qquad
 n=\max\{k: \ c_k=a_{p_{n_i}}\,\},
\end{equation}
because then $c_n=c_{n+1}$ or $c_n=a_j$ for some $j\not\in\{p_{n_i}:\ i\in\mathbb N\,\}$. In the latter case, if $a_j\le r_j^a$, then $c_n\le r_j^a\le r_n^c$. If, however, $a_j>r_j^a$, then denoting $i:=\max\{s:\ a_{p_{n_s}}\ge a_j\,\}$, we get $a_{p_{n_{i+1}}}< a_j$ and, because the $D$-gap $(r_{p_{n_i}}^a,\,a_{p_{n_i}})$ is dominating,
$$
a_j\,-\,r_j^a\ \le\ a_{p_{n_i}}\,-\,r_{p_{n_i}}^a\ \overset{\eqref{uno}}{\le}\ N_{i+1}a_{p_{n_{i+1}}}.
$$
Therefore,
$$
c_n\ =\ a_j\ \le\ r_j^a\,+\, N_{i+1}a_{p_{n_{i+1}}}\ <\ r_j^a\,+\,\sum_{t>i}N_ta_{p_{n_t}}\ =\ r_n^c.
$$
It remains to consider indices $n$ such that \eqref{tres} holds. In this case $c_n=a_{p_{n_i}}$ and
$$
c_n\ =\ a_{p_{n_i}}\ \overset{\eqref{uno}}{\le} \ \ r_{p_{n_i}}^a\,+\, N_{i+1}a_{p_{n_{i+1}}}\ <\ r_{p_{n_i}}^a\,+\,\sum_{j>i}N_ja_{p_{n_j}}\ =\ r_n^c.
$$
\end{proof}
\end{proposition}

In \cite[Theorem 16]{BPW} the authors proved the following theorem.
\begin{proposition}\label{rozkladnadwacantory}
There exists a Cantorval $D=A(a_n)$ (specifically, the Guthrie Nymann's Cantorval) such that for any partition of $(a_n)$ into two infinite sequences $(y_n)$ and $(z_n)$ both $A(y_n)$ and $A(z_n)$ are Cantor sets. 
\end{proposition}

\begin{remark}
Note that in \cite{GM23} the authors showed that for $A=A(2^m,\ldots,4,3,2;\frac{1}{4})$ with any $m\in\mathbb{N}$ the set of all points with the unique representation is dense in $A$. In the same paper it was showed that if the set of all points of an achievement set $A(a_n)$ which have the unique representation is dense in $A(a_n)$, then for any partition of $(a_n)$ into two infinite sequences $(y_n)$ and $(z_n)$ both $A(y_n)$ and $A(z_n)$ are Cantor sets. Hence $A$ also has the property of the set $D$ in the Proposition \ref{rozkladnadwacantory}.
\end{remark}

\begin{example}
Note that one can find an interval-filling sequence $(a_n)$ such that for any its decomposition into two infinite subsequences $(y_n)$ and $(z_n)$ both sets $A(y_n)$ and $A(z_n)$ are Cantor sets. Indeed, let $a_n=\frac{1}{2^n}$, then $A(a_n)=[0,1]$. Since $a_n=r_n$ for all $n\in\mathbb{N}$, any its infinite, non-cofinite subsequence is fast convergent.
\end{example}

The next lemma is the Lemma 9 from \cite{BP} and will be crucial for our next theorem, but we need to introduce one more symbol, used in the lemma below and in the proof of the next theorem. If $S$ is a finite subset of $\mathbb R$, then we define $\delta(S):=\min\{|s-t|:\ s,t\in S, s\ne t\,\}$.

\begin{lemma}
\label{rozbicie}
If all terms of the series $\sum a_n$ are positive and $r_n<\delta(F_n)$ for infinitely many $n$, then $A(a_n)$ is a Cantor set.
\end{lemma}
\begin{proof}
Take any $n$ such that $r_n<\delta(F_n)$. We have
$$A(a_k)=F_n+E_n\subset F_n+[0,r_n]=\bigcup_{i=1}^{m_n}[f_i^n,f_i^n+r_n]$$
and, by the assumption that $r_n<\delta(F_n)$, we get $[f_i^n,f_i^n+r_n]\cap [f_j^n,f_j^n+r_n]=\emptyset$ for $i,j\in\{1,\ldots, m_n\}$, $i\neq j$. Hence $A(a_k)$ has no interval longer that $r_n$. But $\lim\limits_{n\to\infty}r_n=0$, so $A(a_k)$ has an empty interior.
\end{proof}

In the paper \cite{PWT18} it was shown that every central Cantor set is the algebraic sum of two central Cantor sets of Lebesgue measure zero. The following theorem generalizes that result.

\begin{theorem}
Every infinite achievable set is the algebraic sum of two achievable Cantor sets of Lebesgue measure zero.
\end{theorem}

\begin{proof}
Let $A(a_n)$ be an infinite achievement set. Define $n_0 = 1$ and $y_1=a_1$. Then $F^y_1=\{0,a_1\}$, so $\delta(F_1^y)=a_1$. One can find $n_1$ such that $r_{n_1}<a_1$. We define $z_1=a_2, z_2=a_3, \ldots, z_{n_1-1}=a_{n_1}$. Denote $F_{n_1-1}^z=A\big((z_k)_{k=1}^{n_1-1}\big)=A\big((a_k)_{k=2}^{n_1}\big)$. Let $n_2>n_1$ be such that $r_{n_2}<\frac{1}{2\cdot2^{n_1-1}}\cdot\delta(F_{n_1-1}^z)$. We put $y_2=a_{n_1+1}, y_3=a_{n_1+2}, \ldots, y_{n_2-n_1+1}=a_{n_2}$. In the next step we consider $F_{n_2-n_1+1}^y=A\big((y_k)_{k=1}^{n_2-n_1+1}\big)$ and find $n_3>n_2$ such that $r_{n_3}<\frac{1}{3\cdot2^{n_2-n_1+1}}\cdot\delta(F_{n_2-n_1+1}^y)$ to define consecutive terms of $(z_k)$.

Generally, assume that we have defined $n_k$ for some $k \in \N$ and we have used all of the first $n_k$ terms of the sequence $(a_n)$ to define sequences $y$ and $z$. If $k$ is odd, then we take $n_{k+1} > n_k$ such that $r_{n_{k+1}}<\frac{1}{(k+1)\cdot2^{p_k}}\cdot\delta(F_{p_k}^z)$, where $p_k=\sum_{i=0}^\frac{k-1}{2} (n_{2i+1}-n_{2i})$ is the number of elements of the sequence $z$ defined so far. Then we define consecutive terms of the sequence $y$ as $a_{n_k+1}, a_{n_k+2},  \ldots, a_{n_{k+1}}.$
If $k$ is even, then we take $n_{k+1} > n_k$ such that $r_{n_{k+1}}<\frac{1}{(k+1)\cdot2^{q_k}}\cdot\delta(F_{q_k}^y)$, where $q_k=1+\sum_{i=1}^\frac{k}{2} (n_{2i}-n_{2i-1})$ is the number of elements of the sequence $y$ defined so far. Then we define consecutive terms of the sequence $z$ as $a_{n_k+1}, a_{n_k+2},  \ldots, a_{n_{k+1}}.$

Note that both $(y_k)$ and $(z_k)$ satisfy the assumptions of Lemma \ref{rozbicie}. Thus, $A(y)$ and $A(z)$ are Cantor sets. Moreover, for even $k$ we have $A(y) \subset F_{q_k}^y +[0,r^y_{q_k}]$, where $r^y_i$ is the $i$-th remainder for the series $\sum y_n$. Since $|F_{q_k}| \leq 2^{q_k}$ and $$r^y_{q_k}< r_{n_{k+1}} < \frac{1}{(k+1)\cdot2^{q_k}}\cdot\delta(F_{q_k}^y),$$
we have
$$\lambda(A(y)) \leq \lambda(F_{q_k}^y +[0,r^y_{q_k}]) \leq |F_{q_k}^y| \cdot r^y_{q_k} < 2^{q_k} \cdot \frac{1}{(k+1)\cdot2^{q_k}}\cdot\delta(F_{q_k}^y) \leq \frac{a_1}{k+1} \to 0$$
where $\lambda$ denotes the Lebesgue measure.
Similarly, we show that the measure of $A(z)$ is zero.

Observe that $y \cup z$ is a decomposition of the sequence $(a_n)$, and so $A(y) + A(z) = A(a_n)$.

\end{proof}
\begin{remark} 
It is known (see \cite{AI}) that the achievement set of a fast convergent series is a central Cantor set.
Note that if $(a_n)$ in the above theorem is fast convergent, then also $y$ and $z$ are fast convergent, so $A(y)$ and $A(z)$ are central Cantor sets.
\end{remark}
\begin{remark}
From \cite[Proposition 2.11.]{Now1} we can infer something even stronger for an achievable set which is an interval. Namely it can be presented as a sum of two central Cantor sets (that is, the achievement sets of fast convergent series) of Hausdorff dimension zero.
\end{remark}

\begin{remark}\label{infinity}
In \cite[Example 1]{Nymannlin} the author gave an example of a Cantor set $C$ such that for any $k\in\mathbb{N}$ the sum $\underbrace{C+\ldots+C}_{k \ \text{times}}$ is not an interval (actually, it can be showed that this sum is a Cantor set for any $k$). He also characterized the situation by the condition $\limsup\frac{a_n}{r_n}=\infty$.
Note that if $A(a_n)$ is a multi-interval set, then the inequality $a_n>r_n$ holds for finitely many $n$, so $\limsup\frac{a_n}{r_n}<\infty$. Hence there is no such construction for that case.
\end{remark}

\begin{theorem}
\label{immortal}
There is an achievable Cantorval such that the algebraic sum of any finite number of copies of it remains a Cantorval.
\end{theorem}
\begin{proof}
It is well-known that for any infinite achievement set $A(a_n)$ the algebraic sum of any finite number of copies of $A(a_n)$ is not a multi-interval set if and only if $\limsup\frac{a_n}{r_n}=+\infty$. Moreover, if $\limsup \frac{a_n}{r_n}<+\infty$, then the algebraic sum of sufficiently many copies of $A(a_n)$ is an interval.

Thus, in order to prove the theorem, it suffices to construct a convergent GF series $\sum a(m,k,q)$ such that $A(a_n)$ is a Cantorval and $\limsup_n\frac{a_n}{r_n}=+\infty$. To do that, take $m_n:=n+1$ and $c_n:=\frac32$ for all $n\in\mathbb N$. Then, by Theorem \ref{nmg2}, there is a convergent GF series $\sum a(m,k,q)$ with $A(a_n)$ being a Cantorval and satisfying additionally
\begin{equation}
\label{imm1}
(s_{n+1}+m_{n+1})q_{n+1}\ <\ \frac32\,q_n \qquad \text{ for all $n$}
\end{equation}
and
\begin{equation}
\label{imm2}
\frac32\ <\ \left(1\,-\,\frac34\right)(s_n+m_n) \qquad \text{ for all $n$.}
\end{equation}
Then, as in the proof of the Theorem \ref{nmg2},
$$
\frac{(s_{n+1}+m_{n+1})q_{n+1}}{(s_n+m_n)q_n}\ <\ \frac14 \qquad \text{ for all $n$}
$$
and
$$
r_{K_n}\,=\,\sum_{i>n}(s_i+m_i)q_i\ \overset{\eqref{imm2}}{<}\ \sum_{j=0}^\infty(s_{n+1}+m_{n+1})q_{n+1}\frac1{4^j}\ =\ \frac43(s_{n+1}+m_{n+1})q_{n+1}\ \overset{\eqref{imm1}}{<} 2q_n
$$
for all $n$. Thus,
$$
\limsup\limits_{n\to \infty}\frac{a_n}{r_n}\ \ge\ \limsup\limits_{n\to \infty}\frac{a_{K_n}}{r_{K_n}}\ =\ \lim\limits_{n\to\infty}\frac{m_n}2\ =\ +\infty.
$$
\end{proof}

\begin{example}\label{noCantorval}
Let $C=A(\frac{1}{m^n})$, $m\in\mathbb{N}$, $m\geq 3$. Then $C_k=\underbrace{C+\ldots+C}_{k \ \text{times}}$ is a Cantor set for every $k<m-1$, while $C_k$ is an interval for $k\geq m-1$.
\end{example}

\begin{theorem}\label{tkmp}
Let $m,p\in(\mathbb{N}\setminus \{1\}) \cup \{\infty\}$, $p\geq m$. There exists an achievable Cantor set $C$ such that $C_k:=\underbrace{C+\ldots+C}_{k \ \text{times}}$ is
\begin{itemize}
\item a Cantor set for every $k< m$;
\item a Cantorval for each $k$ such that $k\geq m$ and $k<p$;
\item an interval for all $k \in \N$ such that $k\geq p$.
\end{itemize}
\end{theorem}
\begin{proof}
If $m=p\in \N$, then we take $C=A(\frac{1}{(m+1)^n})$ as in Example \ref{noCantorval}. If $m=p = \infty$, then we take $C$ from Remark \ref{infinity}.

Suppose that $6<m < p \leq \infty$.
Define a sequence $(p_n)$ in the following way. If $p < \infty$, then put $p_n:=p$ for all $n \in \N$, and if $p = \infty$, then put $p_n = m+n$.
We will now define a fast convergent sequence $a=(a_n)$ such that $A(a_n)$ satisfies the assertion of the theorem. First, we will define first $k_1$ terms of the sequence for some $k_1 \in \N$. We will define these terms, using backward induction and starting from $a_{k_1}$. For all $n \in \N$ we will define $a_n$ in such a way that $a_n=q_nr_n$ for all $n\in\N$, where $q_n > 1$. When the induction is over for some $a_{k_1-j}$, we will put $k_1 -j =1$, and so $k_1 := j+1$.

Take an arbitrary number $S_1 > 0$. We will define a sequence $(a_n)$ in such a way that $r_{k_1} = S_1$. For the convenience, although we do not have all the terms in the sequence, we will use the notation with remainders $r_{k_i}$ instead of $S_i$.

Put $q_{k_1} := p_1$ and $a_{k_1} := p_1S_1 = q_{k_1}r_{k_1}$.
Observe that, we can now define $r_{k_1-1}$ as $r_{k_1} +a_{k_1} = r_{k_1}+q_{k_1}r_{k_1} = r_{k_1}(1+q_{k_1})$.
Generally, when we know $r_i$ and $a_i=q_ir_i$ for some $i > 1$, $q_i > 0$, then we have \begin{equation}\label{r}
r_{i-1} = r_i + a_i = r_i (1+q_i).
\end{equation}


Suppose that for some $i \geq 0$ we have defined $q_{k_1-i} > 1$
and $a_{k_1-i} := q_{k_1-i} r_{k_1-i}$.
Put
\begin{equation} \label{q}
q_{k_1-(i+1)}:=\frac{(\lfloor \frac{m}{2}\rfloor-1 )q_{k_1-i}}{(1+q_{k_1-i})}+\frac{mr_{k_1}}{2r_{k_1-(i+1)}}.
\end{equation}
Since $m > 6$, we clearly have $q_{k_1-(i+1)} >1.$
Now, put
$a_{k_1-(i+1)} := q_{k_1-(i+1)} r_{k_1-(i+1)}$.

We continue this procedure until we define $q_{k_1-j_1}$ and $a_{k_1-j_1}$, where $j_1=\lceil \frac{2p_1}{m} \rceil-1.$ Then, we put $k_1:=j_1+1$, and thus we have defined $a_1,a_2, a_3,  \ldots, a_{k_1}.$

Suppose that for some $n \in \N$ we have defined $k_n \in \N$, $S_n > 0$, $(q_i)_{i=1}^{k_n}$ and a sequence $(a_i)_{i=1}^{k_n}$. We will define $k_{n+1}>k_n$, a sequence $(a_i)_{i=k_n+1}^{k_{n+1}}$ and $S_{n+1} > 0$ in such a way that $r_{k_{n+1}}=S_{n+1}$. Assume that we have $S_{n+1}>0$ and $k_{n+1}>k_n$ (their actual value will be established later).
Put $q_{k_{n+1}}: = p_{n+1}$ and $a_{k_{n+1}} := q_{k_{n+1}} S_{n+1} = q_{k_{n+1}} r_{k_{n+1}}$.


Suppose that for some $i \geq 0$ we have defined $q_{k_{n+1}-i} > 1$ and $a_{k_{n+1}-i} := q_{k_{n+1}-i} r_{k_{n+1}-i}$.
Put $$q_{k_{n+1}-(i+1)}:=\frac{(\lfloor \frac{m}{2}\rfloor-1 )q_{k_{n+1}-i}}{(1+q_{k_{n+1}-i})}+\frac{mr_{k_{n+1}}}{2r_{k_{n+1}-(i+1)}}.$$
We have $q_{k_1-(i+1)} >1.$ Now, put
$a_{k_{n+1}-(i+1)} := q_{k_{n+1}-(i+1)} r_{k_{n+1}-(i+1)}$.

We continue this procedure until we define $q_{k_{n+1}-j_{n+1}}$ and $a_{k_{n+1}-j_{n+1}}$, where $j_{n+1}=\lceil \frac{2p_{n+1}}{m} \rceil-1.$

Now, put $q_{k_{n+1}-j_{n+1}-1}:=m-3$ and $a_{k_{n+1}-j_{n+1}-1}:= q_{k_{n+1}-j_{n+1}-1}r_{k_{n+1}-j_{n+1}-1}.$
For $i> 1$ define
$q_{k_{n+1}-j_{n+1}-i}:= m-\frac{1}{2}$ and $a_{k_{n+1}-j_{n+1}-i}:= q_{k_{n+1}-j_{n+1}-i}r_{k_{n+1}-j_{n+1}-i}$ until $t_{n+1}> 1$ such that

\begin{equation}\label{miara}
\frac{m^{j_{n+1}+t_{n+1}+1}}{(1+q_{k_{n+1}})\cdot(1+q_{k_{n+1}-1})\cdot \ldots \cdot (1+q_{k_{n+1}-j_{n+1}-t_{n+1}})}<\frac{1}{2}.
\end{equation}
We will find such $t_{n+1}$, because for $i > 1$ we have
$$\frac{m^{j_{n+1}+i+1}}{(1+q_{k_{n+1}})\cdot(1+q_{k_{n+1}-1})\cdot \ldots \cdot (1+q_{k_{n+1}-j_{n+1}-i})} $$$$= \frac{m^{j_{n+1}+2}}{(1+q_{k_{n+1}})\cdot(1+q_{k_{n+1}-1})\cdot \ldots \cdot (1+q_{k_{n+1}-j_{n+1}-1})}\cdot\frac{m^{i-1}}{(1+q_{k_{n+1}-j_{n+1}-2})\cdot \ldots\cdot(1+q_{k_{n+1}-j_{n+1}-i})}$$$$= \frac{m^{j_{n+1}+2}}{(1+q_{k_{n+1}})\cdot(1+q_{k_{n+1}-1})\cdot \ldots \cdot (1+q_{k_{n+1}-j_{n+1}-1})}\cdot\frac{m^{i-1}}{(m+\frac{1}{2})^{i-1}}\xrightarrow{i\to\infty} 0. $$

We put $k_{n+1}:=k_n+j_{n+1}+t_{n+1}+1$, and $S_{n+1} : =\frac{S_n}{(1+q_{k_{n}+1})\cdot \ldots\cdot(1+q_{k_{n+1}})}.$ Since
$$r_{k_{n}} = r_{k_n+1}+a_{k_n+1}=(1+q_{k_n+1})r_{k_{n}+1}=(1+q_{k_n+1})\cdot(r_{k_n+2}+a_{k_{n+2}})=(1+q_{k_n+1})(1+q_{k_n+2})r_{k_n+2}$$$$= \ldots=(1+q_{k_n+1})\cdot \ldots\cdot(1+q_{k_{n+1}})r_{k_{n+1}},$$
we have
$$r_{k_{n+1}}=\frac{r_{k_n}}{(1+q_{k_n+1})\cdot \ldots\cdot(1+q_{k_{n+1}})} = \frac{S_n}{(1+q_{k_{n}+1})\cdot \ldots\cdot(1+q_{k_{n+1}})} =S_{n+1}.$$

Thus, we have inductively defined sequences $(q_i)$, $(k_n)$, $(S_n)$ and $(a_i)$ in such a way that $a_{i}=q_ir_i$ for all $i \in \N$ and for all $n \in \N$ $r_{k_n}=S_n$.

Consider the set $C:=A(a_n)$. It follows from the construction that for all $i\in \N$, $p\geq q_i >1$, thus $a_{i} = q_i r_i > r_i$ for all $i \in \N$, so the sequence $(a_i)$ is fast convergent, which implies that $C$ is a Cantor set. We will now examine sets $C_k $ for $k > 1$. First, observe that $C_k = A(x_n)$, where $(x^{(k)}_n) = (a_i;k)$ that is the sequence in which every term $a_i$ is repeated $k$ times. Denote by $R_n^{(k)}$ a $n$-th remainder of the sequence $(x^{(k)}_n)$, that is,
$R_n^{(k)} = \sum _{i=n+1}^\infty x_i^{(k)}$.
In particular, for $n \in \N$,
$R_{k\cdot n}^{(k)} = \sum _{i=n+1}^\infty ka_i =kr_n.$ Also, let $F^{(k)}_n := \{\sum_{i=1}^n \ve_ix_i^{(k)}\colon \ve_i \in \{0,1\}\}.$

By Kakeya's Theorem we know that $C_k$ is an interval if and only if $R_n^{(k)} \geq x^{(k)}_n$ for all $n \in \N$. For $n$ which are not divisible by $k$ we have $R_n^{(k)} \geq x^{(k)}_{n+1} =x^{(k)}_n$. For indices of the form $n \cdot k$ for some $n \in \N$ we have
$R_{k\cdot n}^{(k)} =kr_n$ and $x_{k\cdot n} = a_n$. So $C_k$ is an interval if and only if
$a_n \leq kr_n$ for all $n \in \N$. From the construction we have $a_{k_n} = p_nr_{k_n}$ for any $n \in \N$. If $p \in \N$, then $p_n = p$ for all $n$, and hence there is infinitely many $n$ such that $a_n = pr_n$ and for all $n \in \N$, $a_n \leq pr_n$. Hence $C_k$ is an interval if and only if $k \geq p$. Moreover, $C_k$ is not a finite union of intervals for $k<p$.
If $p = \infty$, then $p_n = m+n$ for all $n\in \N$, and thus for any $k \in \N$ there is infinitely many $n$ such that $a_n > kr_n$, so $C_k$ is not a finite union of intervals for any $k \in \N$.

We will now show that $C_k$ is a Cantor set for $k < m$. Since $0 \in C$, we have $C_k \subset C_{k+1}$ for all $k\in \N$, so it suffices to show that $C_{m-1}$ is a Cantor set.
First, observe that $C_{m-1} \subset F^{(m-1)}_{n} + [0,R^{(m-1)}_n]$, for all $n \in \N$. In particular,
$C_{m-1} \subset F^{(m-1)}_{(m-1)\cdot k_n} + [0, (m-1)r_{k_n}]$ for $n \in \N$. Therefore, by (\ref{miara}), we have
$$\lambda(C_{m-1}) \leq \lambda(F^{(m-1)}_{(m-1)\cdot k_n} + [0, (m-1)r_{k_n}]) \leq |F^{(m-1)}_{(m-1)\cdot k_n}|\cdot (m-1)r_{k_n} \le \frac{m^{k_n}\cdot(m-1) r_{k_1}}{(1+q_{k_1+1})\cdot (1+q_{k_1+2})\cdot  \ldots \cdot (1+q_{k_n})} $$$$= m^{k_1}(m-1)r_{k_1}\cdot \frac{m^{k_2-k_1}}{(1+q_{k_1+1})\cdot  \ldots \cdot (1+q_{k_2})} \cdot  \ldots \cdot \frac{m^{k_n-k_{n-1}}}{(1+q_{k_{n-1}+1})\cdot  \ldots \cdot (1+q_{k_n})} \leq m^{k_1}(m-1)r_{k_1}\cdot \left( \frac{1}{2} \right) ^ {n-1} \xrightarrow{n \to \infty} 0$$
where $\lambda$ is the Lebesgue measure. Since $C_{m-1}$ has measure zero, it has to be a Cantor set.

Finally, we will show that $C_m$ is a Cantorval, and thus also $C_k$ is a Cantorval for $m \leq k < p$.

We will use Proposition \ref{charint}. We are going to define inductively a sequence of sets $(D_n)_{n\in\mathbb N}$ such that for any $n \in \N$
\begin{itemize}
\item[($\alpha_n$)] \ $D_n\,\subset \, F^{(m)}_{mk_n}$;
\item[($\beta_n$)] \ the distance between any two consecutive points of $D_n$ does not exceed $\frac{mr_{k_n}}{2}$;
\item[($\gamma_n$)] \ $\max D_n - \min D_n \geq (m-2) \sum_{i=1}^{k_n} a_i$.
\end{itemize}
Define
$$L:= \{1,2, \ldots, \lfloor\frac{m}{2} \rfloor \},$$
$$R:=\{\lfloor\frac{m}{2} \rfloor,\lfloor\frac{m}{2} \rfloor +1, \ldots, m-1 \},$$
$$H^1_1:=\{\sum_{i=1}^{k_1} h_ia_i\colon \forall_{i\leq k_1}\,\, h_i \in
L\},$$
$$H^2_1:=\{\sum_{i=1}^{k_1} h_ia_i\colon \forall_{i\leq k_1}\,\,h_i \in
R\},$$
$$H^3_1:=\{\sum_{i=1}^{k_1} h_ia_i\colon \exists_{1<j\leq k_1}\,\, \left( h_j \in
L+\lfloor \frac{m}{2} \rfloor -1\wedge h_1 \in L-1 \wedge \forall_{1<i <j} \,\, h_i \in L+\lfloor \frac{m}{2} \rfloor-2 \wedge \forall_{i > j} \,\, h_i \in L \right) \},$$
$$H^4_1:=\{\sum_{i=1}^{k_1} h_ia_i\colon \exists_{1\leq j< k_1}\,\, \left( h_j \in
R+1 \wedge h_{k_1} \in R -\lfloor \frac{m}{2} \rfloor +1 \wedge \forall_{i <j} \,\, h_i \in R \wedge \forall_{k_1 > i > j} \,\, h_i \in R -\lfloor \frac{m}{2} \rfloor +2 \right) \},$$
$$D_1:=H^1_1\cup H^2_1\cup H^3_1 \cup H^4_1.$$

First, observe that $D_1 \subset F^{(m)}_{mk_1}$, because $F^{(m)}_{mk_1} = \{\sum_{i=1}^{k_1} h_i a_i\colon h_i \in \{0,1, \ldots,m\}\}.$

Now, we will prove $(\beta_1)$. We will do it, proving that for every but one $h \in D_1$, there is $g \in D_1$ such that $g>h$ and $g-h \leq \frac{mr_{k_1}}{2}.$

First, we will prove the following fact. \\
{\bf Claim 1} For any $h=\sum_{i=1}^{k_1} h_ia_i \in F^{(m)}_{mk_n}$ if $g=\sum_{i=1}^{k_1} g_ia_i \in F^{(m)}_{mk_n}$ is such that for some $j \in \{2,3,  \ldots, k_1\}$,
$$g_i := \begin{cases}
h_j - \lfloor \frac{m}{2} \rfloor +1 \;&\text{ if }\; i = j \\
h_{j-1}+1 \;&\text{ if }\; i=j-1\\
h_i \;&\text{ for the remaining }\; i,%
\end{cases} $$
then $g-h = \frac{mr_{k_1}}{2}.$

Using (\ref{r}) and (\ref{q}), we obtain
$$g-h = \sum_{i=1}^{k_1} (g_i-h_i)a_i = (-\lfloor \frac{m}{2} \rfloor +1)a_j + a_{j-1} = (-\lfloor \frac{m}{2} \rfloor +1)q_jr_j + q_{j-1}r_{j-1} $$$$\stackrel{(\ref{r}), (\ref{q})}{=} \frac{(-\lfloor \frac{m}{2} \rfloor +1)q_jr_{j-1}}{1+q_j} + \frac{(\lfloor \frac{m}{2}\rfloor-1 )q_{j}r_{j-1}}{1+q_j}+\frac{mr_{k_1}}{2} =\frac{mr_{k_1}}{2},$$
which proves Claim 1.

Now, we will prove \\
{\bf Claim 2} For any $h=\sum_{i=1}^{k_1} h_ia_i \in F^{(m)}_{mk_n}$ if $g=\sum_{i=1}^{k_1} g_ia_i$ is such that

$$g_i := \begin{cases}
h_1 - 1 \;&\text{ if }\; i = 1 \\
h_{k_1} + \lfloor \frac{m}{2}\rfloor \;&\text{ if }\; i=k_1\\
h_i+ \lfloor \frac{m}{2}\rfloor-2 \;&\text{ for the remaining }\; i,%
\end{cases}%
$$
%
then $g-h \in (0, \frac{mr_{k_1}}{2}].$

For $j \in \{1,2,  \ldots, k_1-2\}$ define $g^j=\sum_{i=1}^{k_1} g^j_ia_i$, where
$$g^j_{i} := \begin{cases}
h_{k_1} +1 \;&\text{ if }\; i=k_1\\
h_i\;&\text{ if }\; k_1>i >k_1-j \\
h_{k_1-j} + \lfloor \frac{m}{2}\rfloor -1\;&\text{ if }\; i=k_1-j \\
g_i\;&\text{ if }\; i<k_1-j.%
\end{cases}%
$$
%
We also define $g^{k_1-1}$ such that
$$g^{k_1-1}_{i} := \begin{cases}
h_{k_1} +1 \;&\text{ if }\; i=k_1\\
h_i\;&\text{ if }\; i<k_1.\\
\end{cases}%
$$
Observe that
$g^1_{k_1} = h_{k_1} +1 = g_{k_1} - \lfloor \frac{m}{2} \rfloor +1$, $g^1_{k_1-1} = h_{k_1-1} + \lfloor \frac{m}{2}\rfloor -1= g_{k_1-1}+1$ and $g^1_i = g_i$ for the remaining $i$. Similarly, for $j \in \{2, 3,  \ldots, k_1-2\}$ we have $g^j_{k_1-j+1} = h_{k_1-j+1} = g^{j-1}_{k_1-j+1} - \lfloor \frac{m}{2} \rfloor+1$, $g^j_{k_1-j} = h_{k_1-j} + \lfloor \frac{m}{2}\rfloor -1 = g^{j-1}_{k_1-j} + 1$ and $g^{j}_i = g^{j-1}_i$ for the remaining $i$. Also, $g^{k_1-1}_{2} = h_{2} = g^{k_1-2}_2 - \lfloor \frac{m}{2} \rfloor+1$, $g^{k_1-1}_1 = h_{1} = g^{k_1-2}_{1} + 1$ and $g^{k_1-1}_i = g^{k_1-2}_i$ for the remaining $i$.
Hence, by Claim 1, for $j \in \{1,2,  \ldots, k_1-1\}$ we have $g^j - g^{j-1} = \frac{mr_{k_1}}{2}$ (where $g^0:=g$).
Therefore,
$$g^{k_1-1} - g = (k_1-1) \cdot \frac{mr_{k_1}}{2}.$$ On the other hand, $g^{k_1-1}$ is such that $g^{k_1-1}_{k_1} = h_{k_1}+1$ and $g^{k_1-1}_i = h_i$ for $i < k_1$. So,
$$g^{k_1-1} - h = a_{k_1} = p_1r_{k_1}.$$
Since $k_1 = \lceil \frac{2p_1}{m} \rceil,$ we have
$$g - h = g^{k_1-1} -h - (g^{k_1-1}-g)= p_1r_{k_1} - (k_1-1) \cdot \frac{mr_{k_1}}{2} =
p_1r_{k_1} - (\lceil \frac{2p_1}{m}-1 \rceil)\cdot\frac{mr_{k_1}}{2} $$$$\leq r_{k_1}\left(p_1- (\frac{2p_1}{m}-1) \cdot\frac{m}{2} \right) = \frac{m{r_{k_1}}}{2}.$$
We also have
$$g - h > r_{k_1}\left(p_1- (\frac{2p_1}{m}) \cdot\frac{m}{2} \right) =0.$$
Thus, $g - h \in (0, \frac{mr_{k_1}}{2}],$ which finishes the proof of Claim 2.

For $j \in \{2,  \ldots, k_1\}$ denote by $b^j$ the sequence of $k_1$ terms such that
$$b^j_{i} := \begin{cases}
(-\lfloor \frac{m}{2} \rfloor +1) \;&\text{ if }\; i=j\\
1\;&\text{ if }\; i=j-1 \\
0\;&\text{ for the remaining }\; i.%
\end{cases}%
$$
In Claim 1. we proved that
for any $h = \sum_{i=1}^{k_1}h_ia_i \in F^{(m)}_{mk_n}$ if $g_i = h_i +b^j_i$ for some $j\in \{2,  \ldots, k_1\}$ and $g = \sum_{i=1}^{k_1} g_ia_i$, then $g-h = \frac{mr_{k_1}}{2}.$

By $B$ denote the sequence of $k_1$ terms such that
$$B_{i} := \begin{cases}
-1 \;&\text{ if }\; i=1\\
\lfloor \frac{m}{2} \rfloor\;&\text{ if }\; i=k_1 \\
(\lfloor \frac{m}{2} \rfloor -2) \;&\text{ for the remaining }\; i.%
\end{cases}%
$$
From Claim 2. we know that
for any $h = \sum_{i=1}^{k_1}h_ia_i \in F^{(m)}_{mk_n}$ if $g_i = h_i +B_i$ and $g = \sum_{i=1}^{k_1} g_ia_i$, then $g-h \in (0, \frac{mr_{k_1}}{2}].$

Now, we will show that for every (but one) $h \in D_1$ there is $g =\sum_{i=1}^{k_1}g_ia_i \in D_1$ such that $g-h\in (0, \frac{mr_{k_1}}{2}]$. Let $h \in D_1$. Consider the cases.

1. $h=\sum_{i=1}^{k_1} h_ia_i \in H^1_1$, that is, $h_i \in L$ for all $i\leq k_1$.
Consider the subcases.

1.1. $h_{k_1} \neq \lfloor \frac{m}{2} \rfloor$. Let $g_i= h_i + B_i$ for all $i$. Then $g_{k_1} \in L - 1 + \lfloor \frac{m}{2} \rfloor$, $g_1 \in L-1$ and $g_i \in L +\lfloor \frac{m}{2} \rfloor -2$ for the remaining $i$, so $g \in H^1_3$ (with $j = k_1$) and, by Claim 2, $g-h \in (0, \frac{mr_{k_1}}{2}]$.

1.2. $h_{k_1} = \lfloor \frac{m}{2} \rfloor$ and there is $i < k_1$ such that $h_i < \lfloor \frac{m}{2} \rfloor$. Let $1<j\leq k_1$ be such that $h_j =\lfloor \frac{m}{2} \rfloor$ and $h_{j-1} < \lfloor \frac{m}{2} \rfloor$. Let $g_i = h_i + b^j_i$ for all $i \leq k_1$. Then $g_j = 1 \in L$, $g_{j-1} = h_{j-1} + 1 \in L$ and $g_i = h_i \in L$ for the remaining $i$, so $g \in H^1_1$. Moreover, by Claim 1, $g-h = \frac{mr_{k_1}}{2}.$

1.3. $h_i = \lfloor \frac{m}{2} \rfloor$ for all $i$.
Let $g_i = h_i + b^{k_1}_i$ for all $i \leq k_1$. Then $g_{k_1} = 1 \in R-\lfloor\frac{m}{2} \rfloor +1 $, $g_{k_1-1} = \lfloor\frac{m}{2} \rfloor +1 \in R+1$ and $g_i = \lfloor\frac{m}{2} \rfloor \in R$ for the remaining $i$, so $g \in H^4_1$ (with $j= k_1-1$). Moreover, by Claim 1, $g-h = \frac{mr_{k_1}}{2}.$

2. $h=\sum_{i=1}^{k_1} h_ia_i \in H^2_1$.
Let $g_i = h_i + b^{k_1}_i$ for all $i \leq k_1$. Then $g_{k_1} \in R-\lfloor\frac{m}{2} \rfloor +1 $, $g_{k_1-1} \in R+1$ and $g_i\in R$ for the remaining $i$, so $g \in H^4_1$ (with $j= k_1-1$). Moreover, by Claim 1, $g-h = \frac{mr_{k_1}}{2}.$

3. $h=\sum_{i=1}^{k_1} h_ia_i \in H^3_1$. So, there is $1<j \leq k_1$ such that $h_j \in L + \lfloor\frac{m}{2} \rfloor-1$, $h_1 \in L-1$, $h_i \in L + \lfloor\frac{m}{2} \rfloor -2$ for all $1<i<j$ and $h_i \in L$ for $i > j$. Let $g_i = h_i + b^{j}_i$ for all $i \leq k_1$. If $j > 2$, then $g_j \in L$ and $g_{j-1} \in L + \lfloor\frac{m}{2} \rfloor-1$, so still $g \in H^3_1.$ If $j = 2$, then for all $i\leq k_1$ we have $g_i \in L$, so $g \in H^1_1$. In both cases $g \in D_1$ and, by Claim 1, $g-h = \frac{mr_{k_1}}{2}.$

4. $h=\sum_{i=1}^{k_1} h_ia_i \in H^4_1$. So, there is $1\leq j < k_1$ such that $h_j \in R +1$, $h_{k_1} \in R-\lfloor\frac{m}{2} \rfloor+1$, $h_i \in R - \lfloor\frac{m}{2} \rfloor +2$ for all $k_1>i>j$ and $h_i \in R$ for $i < j$. Consider the subcases.

4.1. $j > 1$. Let $g_i = h_i + b^{j}_i$ for all $i \leq k_1$. Then $g_j \in R-\lfloor\frac{m}{2} \rfloor+2 $ and $g_{j-1} \in R+1$, so still $g \in H^4_1.$ Moreover, by Claim 1, $g-h = \frac{mr_{k_1}}{2}.$

4.2. $j=1$, $h_{k_1} \neq m-\lfloor\frac{m}{2} \rfloor.$ Let $g_i = h_i +B_i$ for all $i \leq k_1$. Then since $h_{k_1} < m-\lfloor\frac{m}{2} \rfloor$, we have $g_{k_1} \in R$. We also have $g_i \in R$ for $i < k_1$, so $g \in H^2_1$. Moreover, by Claim 2, $g-h \in (0, \frac{mr_{k_1}}{2}].$

4.3. $j=1$, $h_{k_1} = m-\lfloor\frac{m}{2} \rfloor$ and either there is $1< l <k_1$ such that $h_l \neq m - \lfloor \frac{m}{2}\rfloor +1$ or $h_1 \neq m$ (then $l=1$). If $l = k_1-1$, then let $g_i = h_i +b^{k_1}_i$ for all $i \leq k_1$. We have
$g_{k_1} = m - 2\cdot \lfloor \frac{m}{2}\rfloor +1 \in \{1,2\} \subset R-\lfloor \frac{m}{2}\rfloor+1$ and since $h_{k_1-1} < m - \lfloor \frac{m}{2}\rfloor +1$, we get
$g_{k_1-1} \in R- \lfloor \frac{m}{2}\rfloor +2$. Therefore, $g \in H^4_1$.
Now, without loss of generality suppose that $1< l <k_1-1$ is such that $h_{l+1} = m - \lfloor \frac{m}{2}\rfloor +1$. Then let $g_i = h_i +b^{l+1}_i$ for all $i \leq k_1$. We have $g_{l+1} = m - 2\cdot \lfloor \frac{m}{2}\rfloor +2 \in \{2,3\} \subset R-\lfloor \frac{m}{2}\rfloor+2$ and since $h_{j} < m - \lfloor \frac{m}{2}\rfloor +1$, we get
$g_{j} \in R- \lfloor \frac{m}{2}\rfloor +2$. Therefore, $g \in H^4_1$.
If $l=1$, then we can suppose that $h_{2} = m - \lfloor \frac{m}{2}\rfloor +1$ (in the other case we will have the situation as above with $l > 1$). Let $g_i = h_i +b^{2}_i$ for all $i \leq k_1$. We have $g_{2} = m - 2\cdot \lfloor \frac{m}{2}\rfloor +2 \in \{2,3\} \subset R-\lfloor \frac{m}{2}\rfloor+2$ and since $h_{1} < m $, we get
$g_{j} \in R+1$. Therefore, $g \in H^4_1$.
In all of above cases we also have, by Claim 1, $g-h = \frac{mr_{k_1}}{2}$.

This finishes the proof of $(\beta_1)$ as the point $h=\sum_{i=1}^{k_1} h_ia_i \in H^4_1$, where $h_{k_1} = m - \lfloor \frac{m}{2}\rfloor$, $h_1 = m$ and $h_ i = m-\lfloor \frac{m}{2}\rfloor+1$ for the remaining $i$ is the largest number in $D_1$.

As $\sum_{i=1}^{k_1} a_i \in H^1_1 \subset D_1$ and $(m-1)\sum_{i=1}^{k_1} a_i \in H^2_1 \subset D_1$,
we have $\max D_1 \geq (m-1) \sum_{i=1}^{k_1} a_i$ and $\min D_1 \leq \sum_{i=1}^{k_1} a_i \in H^1_1 \subset D_1$. Therefore, $\max D_1 - \min D_1 \geq (m-2) \sum_{i=1}^{k_1},$ which proves $(\gamma_1)$.

Suppose that for some $n \in \N$ we have defined the set $D_n$ satisfying conditions $(\alpha_n), (\beta_n)$ and $(\gamma_n)$.

Define
$$ v_{n+1}: =k_{n+1}-j_{n+1},$$
$$H^1_{n+1}:=\{\sum_{i=v_{n+1}}^{k_{n+1}} h_ia_i\colon \forall_{i\leq k_{n+1}}\,\, h_i \in
L\},$$
$$H^2_{n+1}:=\{\sum_{i=v_{n+1}}^{k_{n+1}} h_ia_i\colon \forall_{i\leq k_{n+1}}\,\,h_i \in
R\},$$
$$H^3_{n+1} :=\{\sum_{i=v_{n+1}}^{k_{n+1}} h_ia_i\colon \exists_{v_{n+1}<j\leq k_{n+1}}\,\, \left( h_j \in
L+\lfloor \frac{m}{2} \rfloor -1\wedge h_{v_{n+1}} \in L-1 \right.$$$$ \left.\wedge \forall_{v_{n+1}<i <j} \,\, h_i \in L+\lfloor \frac{m}{2} \rfloor-2 \wedge \forall_{i > j} \,\, h_i \in L \right) \},$$

$$H^4_{n+1}:=\{\sum_{i=v_{n+1}}^{k_{n+1}} h_ia_i\colon \exists_{v_{n+1}\leq j< k_{n+1}}\,\, \left( h_j \in
R+1 \wedge h_{k_{n+1}} \in R -\lfloor \frac{m}{2} \rfloor +1 \wedge \forall_{i <j} \,\, h_i \in R\right.$$$$ \wedge \left. \forall_{k_{n+1} > i > j} \,\, h_i \in R -\lfloor \frac{m}{2} \rfloor +2 \right) \},$$
$$H_{n+1}: =\bigcup_{i=1}^4 H^i_{n+1},$$
$$G^i = \{0,a_i, 2a_i,  \ldots, ma_i\} \qquad \text{for $i\in\mathbb N$},$$
$$G_{n+1} = \{\sum_{i={k_n+1}}^{v_{n+1}-1} h_i a_i \colon \forall_{i} h_i \in \{0,1, \ldots,m\}\},$$
$$D_{n+1} = D_n + G_{n+1} + H_{n+1}.$$

By the definition, $(\alpha_{n+1})$ is satisfied.

Now, we will prove $(\beta_{n+1})$. First, observe that definition of $a_{v_{n+1}},  \ldots, a_{k_{n+1}}$ is analogous to the definition of $a_1,  \ldots, a_{k_1}$, so we can repeat the reasoning from the proof of $(\beta_1)$ to show that for every but one point $h \in H_{n+1}$ there is $g \in H_{n+1}$ such that $g-h \in (0, \frac{mr_{k_{n+1}}}{2}].$ We also know that $$[\min H_{n+1}, \max H_{n+1}] \supset [\sum_{i=v_{n+1}}^{k_{n+1}} a_i, (m-1)\sum_{i=v_{n+1}}^{k_{n+1}} a_i].$$

Consider the set $G^{v_{n+1}-1} + H_{n+1}.$
For $j \in \{1,2,  \ldots, m\}$ we have
$$ (j-1) a_{v_{n+1}-1} + (m-1) \sum_{i=v_{n+1}}^{k_{n+1}} a_i - \left(ja_{v_{n+1}-1}+\sum_{i=v_{n+1}}^{k_{n+1}} a_i \right) = -a_{v_{n+1}-1} + (m-2)\sum_{i=v_{n+1}}^{k_{n+1}} a_i $$$$=-(m-3)r_{v_{n+1}-1}+(m-2)\sum_{i=v_{n+1}}^{k_{n+1}}a_i = -(m-3)\sum_{i=v_{n+1}}^\infty a_i+(m-2)\sum_{i=v_{n+1}}^{k_{n+1}}a_i = \sum_{i=v_{n+1}}^{k_{n+1}}a_i -(m-3) \cdot\sum_{i=k_{n+1}+1}^\infty a_i $$$$=\sum_{i=v_{n+1}}^{k_{n+1}}a_i -(m-3) r_{k_{n+1}}\geq a_{k_{n+1}}-(m-3)r_{k_{n+1}} = p_{n+1}r_{k_{n+1}}- (m-3)r_{k_{n+1}} > 0. $$
Therefore,
$$\left( (j-1) a_{v_{n+1}-1}+[\sum_{i=v_{n+1}}^{k_{n+1}} a_i, (m-1)\sum_{i=v_{n+1}}^{k_{n+1}} a_i] \right) \cap \left( j a_{v_{n+1}-1}+[\sum_{i=v_{n+1}}^{k_{n+1}} a_i, (m-1)\sum_{i=v_{n+1}}^{k_{n+1}} a_i] \right)\neq \emptyset,$$
so, for any $h \in G^{v_{n+1}-1}+ H_{n+1}$, there is $g \in G^{v_{n+1}-1}+H_{n+1}$ such that $g-h \in (0, \frac{mr_{k_{n+1}}}{2}].$

Suppose that for some $l\in \{ k_{n} + 2,  \ldots, v_{n+1}-1\}$ we have proved that for any $h \in G^l+G^{l+1}+ \ldots+G^{v_{n+1}-1}+ H_{n+1}$ there is $g \in G^l+G^{l+1}+ \ldots+G^{v_{n+1}-1}+H_{n+1}$ such that $g-h \in (0, \frac{mr_{k_{n+1}}}{2}].$
We will prove that for any $h \in G^{l-1}+G^{l}+ \ldots+G^{v_{n+1}-1}+ H_{n+1}$ there is $g \in G^{l-1}+G^{l}+ \ldots+G^{v_{n+1}-1}+H_{n+1}$ such that $g-h \in (0, \frac{mr_{k_{n+1}}}{2}].$

For $j \in \{1,2,  \ldots, m\}$, since $m > 6$ and $j_{n+1}=\lceil 2p_{n+1}{m} \rceil -1 >2,$  using (\ref{r}), we obtain
$$ (j-1) a_{l-1} +ma_l+ \ldots+ma_{v_{n+1}-1} +(m-1) \sum_{i=v_{n+1}}^{k_{n+1}} a_i - \left(ja_{l-1}+\sum_{i=v_{n+1}}^{k_{n+1}} a_i \right) $$$$= -a_{l-1} + (m-2)\sum_{i=l}^{k_{n+1}} a_i +2a_l+2a_{l+1}+ \ldots+2a_{v_{n+1}-1} \geq -(m-\frac{1}{2})r_{l-1}+(m-2)\sum_{i=l}^{k_{n+1}}a_i + 2a_l $$$$
=-(m-\frac{1}{2})\sum_{i=l}^\infty a_i+(m-2)\sum_{i=l}^{k_{n+1}}a_i + 2a_l =-\frac{3}{2}\sum_{i=l}^\infty a_i-(m-2)\sum_{i=k_{n+1}+1}^{\infty}a_i + 2a_l$$$$
= -\frac{3}{2} r_{l-1} -(m-2) r_{k_{n+1}} + 2a_l = -\frac{3}{2}(1+q_l)r_l+2q_lr_l -(m-2) r_{k_{n+1}} $$$$= \frac{1}{2}q_lr_l-\frac{3}{2}r_l - (m-2) r_{k_{n+1}} \geq \frac{1}{2}(m-3)r_l- \frac{3}{2}r_l -(m-2) r_{k_{n+1}} = (\frac{m}{2}-3)r_l- (m-2) r_{k_{n+1}} $$$$=(\frac{m}{2}-3)(1+q_{l+1})(1+q_{l+2})r_{l+2}- (m-2) r_{k_{n+1}} \geq(\frac{m}{2}-3)\cdot4r_{l+2}- (m-2) r_{k_{n+1}} $$$$\geq (2m-12)(p_{n+1}+1)r_{k_{n+1}} - (m-2) r_{k_{n+1}}  \stackrel{m>6}{\geq} 2(p_{n+1}+1)r_{k_{n+1}}- (m-2) r_{k_{n+1}} > 0. $$
Therefore,
$$\left( (j-1) a_{l-1}+[\sum_{i=v_{n+1}}^{k_{n+1}} a_i, m\sum_{i=j}^{v_{n+1}-1}a_i + (m-1)\sum_{i=v_{n+1}}^{k_{n+1}} a_i] \right) \cap \left( j a_{l-1}+[\sum_{i=v_{n+1}}^{k_{n+1}} a_i, m\sum_{i=j}^{v_{n+1}-1}a_i + (m-1)\sum_{i=v_{n+1}}^{k_{n+1}} a_i] \right)\neq \emptyset,$$
so for any
$h \in G^{l-1}+G^{l}+ \ldots+G^{v_{n+1}-1}+ H_{n+1}$ there is $g \in G^{l-1}+G^{l}+ \ldots+G^{v_{n+1}-1}+H_{n+1}$ such that $g-h \in (0, \frac{mr_{k_{n+1}}}{2}].$

By induction we obtain that for any
$h \in G_{n+1}+ H_{n+1}$ there is $g \in G_{n+1}+H_{n+1}$ such that $g-h \in (0, \frac{mr_{k_{n+1}}}{2}].$

Now, let $d_j, d_{j+1} \in D_{n}$ be such that $d_{j+1}>d_j$, and there is no point $f \in D_n$ such that $d_{j+1} > f > d_j$. By $(\beta_n)$, we know that $d_{j+1} - d_j \leq \frac{mr_{k_n}}{2}.$
We have
$$d_j + m\sum_{i=k_n+1}^{v_{n+1}-1}a_i+(m-1) \sum_{i=v_{n+1}}^{k_{n+1}}a_i - \left(d_{j+1} + \sum_{i=v_{n+1}}^{k_{n+1}}a_i\right) $$$$\geq -\frac{mr_{k_n}}{2} + \left( m\sum_{i=k_n+1}^{\infty}a_i -m\sum_{i=v_{n+1}}^{k_{n+1}}a_i - m \sum_{i=k_{n+1}+1}^\infty a_i \right) + (m-2)\sum_{i=v_{n+1}}^{k_{n+1}}a_i = -\frac{mr_{k_n}}{2} + mr_{k_n} -2\sum_{i=v_{n+1}}^{k_{n+1}}a_i - m r_{k_{n+1}}$$$$=
 \frac{mr_{k_n}}{2} - 2\sum_{i=v_{n+1}}^{k_{n+1}}a_i - mr_{k_{n+1}} > \frac{mr_{k_n}}{2} - 2r_{k_n} - \frac{m}{1+p_{n+1}} r_{k_{n+1}-1} >\frac{mr_{k_n}}{2}-3r_{k_n} > 0.$$
Therefore,
$$\left( d_j+\left[\sum_{i=v_{n+1}}^{k_{n+1}} a_i,m\sum_{i=k_n+1}^{v_{n+1}-1}a_i+(m-1) \sum_{i=v_{n+1}}^{k_{n+1}}a_i\right]\right) \cap \left( d_{j+1}+\left[\sum_{i=v_{n+1}}^{k_{n+1}} a_i,m\sum_{i=k_n+1}^{v_{n+1}-1}a_i+(m-1) \sum_{i=v_{n+1}}^{k_{n+1}}a_i\right]\right) \neq \emptyset.$$
Since
$$D_{n+1}\ \cap\ \left[\sum_{i=v_{n+1}}^{k_{n+1}} a_i,m\sum_{i=k_n+1}^{v_{n+1}-1}a_i+(m-1) \sum_{i=v_{n+1}}^{k_{n+1}}a_i\right]\ \subset\ G_{n+1}+H_{n+1}, $$
by arbitrariness of $d_j$, we infer $(\beta_{n+1})$.

To prove $\gamma_{n+1}$ observe that
$$\max D_{n+1} \geq \max D_n +m\sum_{i=k_n+1}^{v_{n+1}-1}a_i+(m-1) \sum_{i=v_{n+1}}^{k_{n+1}}a_i \geq (m-1)\sum_{i=1}^{k_n}a_i + (m-1) \sum_{i=v_{n+1}}^{k_{n+1}}a_i = (m-1)\sum_{i=1}^{k_{n+1}}a_i.$$
Similarly,
$\min D_{n+1} \leq \sum_{i=1}^{k_{n+1}}a_i$. Therefore, we have
$\max D_n -\min D_n \geq (m-2)\sum_{i=1}^{k_{n+1}}a_i.$

Since $\lim\limits_{n\to \infty} (m-2)\sum_{i=1}^{k_{n}}a_i > 0,$ we get that $C_m$ is a Cantorval, by Proposition \ref{charint}.

Now, suppose that $2 \leq m \leq 6$ and $m < p$.
Since $4m > 6$ we know that there is a Cantor set $C$ such that $C_k$ is a Cantor set for $k < 4m$, a Cantorval for $4m\leq k < 4p$ and an interval for $k \geq 4p.$
In particular, $K:=C_4$ is a Cantor set. Denote by $K_k$ an algebraic sum of $k$ copies of $K$. Observe that $K_k = C_{4k}$. So, if $k < m$, then $K_{k}$ is a Cantor set, if $m \leq k < p$, then $K_k$ is a Cantorval and if $k \geq p$, then $K_k$ is an interval, so $K$ is the required Cantor set.
\end{proof}

\section{Decomposition of interval-filling sequence}
In the whole section we will consider an interval-filling sequence $(a_n)$. We divide its terms into two infinite subsequences $(y_n)$ and $(z_n)$, that is, $(y_n)\cup (z_n)=(a_n)$. By $r_n,r_n^{(y)},r_n^{(z)}$ we denote the tails of the sequences $(a_n),(y_n)$ and $(z_n)$ respectively. We assume that all three considered sequences are nonincreasing.

\begin{theorem}\label{nadwaprzedzialy} Let $(a_n)$ be an interval-filling sequence.
If there exists a decomposition $(a_n) = (y_n) \cup (z_n)$ such that both $(y_n)$ and $(z_n)$ are interval-filling, then there exists $k$ such that for each $n\geq k$ the inequality $a_{n-1}+a_n\leq r_{n}$ holds.
\begin{proof}
Without loss of generality we may assume that $y_1=a_1$. Then there exists $k$ such that $y_l=a_l$ for all $l<k$ and $z_1=a_k$. Fix $v\geq k$. We may assume that $a_v=y_j$ for some $j$ (the proof when $a_v\in (z_n)$ is similar). Let $w<v$ be such that $a_w=z_i$ and for every $r\in\{w+1,w+2,\ldots,v-1\}$ we have $a_r\in(y_n)$, that is $a_w$ is the smallest element from $(z_n)$ which is not less than $y_j$. 
By the assumption, we have $z_i\leq\sum_{q>i}z_q=r_i^{(z)}$ and $y_j\leq\sum_{p>j}y_p=r_j^{(y)}$. Note that $(a_n)_{n>v}=(y_p)_{p>j}\cup (z_q)_{q>i}$, since $z_{i+1}\leq y_j\leq z_{i}$. Hence
$$a_{v-1}+a_v\leq z_i+y_j\leq r_i^{(z)}+r_j^{(y)}=r_{v}.$$
\end{proof}
\end{theorem}

\begin{theorem}
Let $(a_n)$ be an interval-filling sequence.
If there exists a decomposition $(a_n) = (y_n) \cup (z_n)$ such that both $(y_n)$ and $(z_n)$ are fast convergent, then for infinitely many $n$ the inequality $a_{n-1}+a_n> r_{n}$ holds.
\begin{proof}
Let $v$ be such that $a_{v-1}\in (y_n)$ and $a_{v}\in (z_n)$ (or vice-versa). Then $a_{v-1}=y_j$ and $a_v=z_i$ for some $i,j$. We have
$$a_{v-1}+a_v=y_j+z_i> r_j^{(y)}+r_i^{(z)}=r_v.$$
\end{proof}
\end{theorem}

\begin{corollary}
Let $(a_n)$ be an interval-filling sequence. Then at most one of the conditions holds
\begin{itemize}
\item there exists a decomposition of $(a_n)$ into two interval-filling sequences;
\item there exists a decomposition of $(a_n)$ into two fast convergent sequences.
\end{itemize}
\end{corollary}

\begin{corollary}
Let $(a_n)$ be an interval-filling sequence. If there exists a decomposition of $(a_n)$ into two interval-filling sequences then $2a_n\leq r_n$ for all large enough $n$. Generally, for $k\geq 2$ if there exists a decomposition of $(a_n)$ into $k$ interval-filling sequences, then $\sum_{i=n-k+1}^{n}a_i\leq r_n$ for all large enough $n$. In particular, $ka_n\leq r_n$ for all large enough $n$. 
\end{corollary}

\begin{corollary}
Let $(a_n)$ be an interval-filling sequence. If there exists a decomposition such that both $(y_n)$ and $(z_n)$ are fast convergent then for infinitely many $n$ the inequality $2a_{n-1}> r_{n}$ holds.
\end{corollary}

So far we considered the necessary conditions for decompositions into interval-filling or fast convergent sequences. Now, we give some sufficient condition to obtain the particular alternating decompositions. We will use the notation $r_n^{(2)}$ for the subtail $a_{n+1}+a_{n+3}+a_{n+5}+\ldots$.

\begin{theorem}
Let $(a_n)$ be an interval-filling sequence. If $2a_{n-1}\leq r_{n}$ for all $n\geq 2$, then there exists a decomposition of $(a_n)$ into two interval-filling sequences.
\begin{proof}
We will show that $(y_n)=(a_{2n-1})$ and $(z_n)=(a_{2n})$ is the proper decomposition. Indeed,
$$2a_{n-1}\leq r_n = r_{n+1}^{(2)}+r_{n}^{(2)}\leq 2r_{n}^{(2)},$$
which means that $a_{n-1}\leq r_{n}^{(2)}$ for all $n$. The inequality for even $n$'s is equivalent to $y_k\leq r_{k}^{y}$ for all $k$, while the case of odd $n$'s gives the inequality $z_k\leq r_{k}^{z}$. Hence both sequences $(y_n)$ and $(z_n)$ are interval-filling.
\end{proof}
\end{theorem}

In a similar way we obtain $k$ alternating interval-filling subsequences $(a_{kn+i})$.

\begin{corollary}
Let $(a_n)$ be an interval-filling sequence and $k\in\mathbb{N}$. If $ka_{n-k+1}\leq r_{n}$ for all $n\geq k$, then there exists a decomposition of $(a_n)$ into $k$ interval-filling sequences.
\end{corollary}
We also have the following theorem.
\begin{theorem}\label{2k-1xn}
Let $(a_n)$ be an interval-filling sequence.
Assume that for $k \in \N$, $(2k-1)a_{n}\leq r_{n}$ for all $n$. Then, there exists a decomposition of $(a_n)$ into $k$ interval-filling sequences.
\end{theorem}
\begin{proof}
Since for all $n \in \N$
$$(k-1)a_n \geq a_{n+1} +  \ldots + a_{n+k-1}$$
and for all $k \leq i \leq 2k-2$
$$r_{n+i}^{(k)} \leq r_{n+k-1}^{(k)},$$
we get 
$$ka_n = (2k-1)a_n - (k-1)a_n \leq r_n-(k-1)a_n =(a_{n+1}+a_{n+2} +  \ldots + a_{n+k-1}) + (r_{n+k-1}^{(k)} + r_{n+k}^{(k)} +  \ldots r_{n+2k-2}^{(k)}) - (k-1)a_n $$$$\leq (k-1)a_n + kr_{n+k-1}^{(k)} - (k-1)a_n = kr_{n+k-1}^{(k)}.$$
So, we decompose $(a_n)$ into subsequences of the form $(a_{kn-j})_n$, where $j \in \{0,1,  \ldots, k-1\}$, and then
$$a_{kn-j} \leq r_{kn-j+k-1}^{(k)} = a_{kn-j+k} + a_{kn-j+2k}+ \ldots.$$
Hence subsequences $(a_{kn-j})_n$ are slowly convergent for all $j \in \{0,1,  \ldots, k-1\}$, and thus they are interval-filling.

\end{proof}
From the above theorem we get that if $3a_{n}\leq r_{n}$ for all $n$ then there exists a decomposition of $(a_n)$ into two interval-filling sequences. However, we can improve this result.
\begin{theorem}
If $(1+\sqrt{3})a_{n}\leq r_{n}$ for all $n$ then there exists a decomposition of $(a_n)$ into two interval-filling sequences.
\end{theorem}
\begin{proof}
For all $n \in \N$ we have
$$(1+\sqrt{3})a_{n}\leq r_{n} = r_n^{(2)} + r_{n+1}^{(2)} =a_{n+1} + r_{n+2}^{(2)} + r_{n+1}^{(2)} \leq 2r_{n+1}^{(2)} + a_{n+1} \leq 2r_{n+1}^{(2)}+\frac{1}{1+\sqrt{3}} r_{n+1} $$$$= 2r_{n+1}^{(2)}+\frac{1}{1+\sqrt{3}}r_{n+1}^{(2)} + \frac{1}{1+\sqrt{3}}r_{n+2}^{(2)} \leq 2r_{n+1}^{(2)}+\frac{2}{1+\sqrt{3}}r_{n+1}^{(2)}.$$
Dividing both sides by $1+\sqrt{3}$, we get
$$a_{n} \leq r_{n+1}^{(2)}.$$
Similarly, as in the proof of Theorem \ref{2k-1xn} we get that subsequences $(a_{2n})$ and $(a_{2n-1})$ are interval-filling.
\end{proof}

Note that there is no need to consider the case of decompositions of sequences into more than two fast-convergent sequences since any subsequence of such a sequence is also fast convergent. Thus, if there exists a decomposition of a sequence into two fast-convergent subsequences, then there exists a decomposition into any number of such sequences.

To sum up and simplify the problem note that the crucial in our consideration is the ratio $q_n:=\frac{a_n}{r_n}$. The class of interval-filling sequences are defined by the condition $q_n\leq 1$ for all $n$. The sufficient condition of having its decomposition into $k$ interval-filling subsequences is that the inequality $q_n\leq\frac{1}{2k-1}$ holds for all $n$. By the necessary condition we know that if the inequality $q_n>\frac{1}{k}$ holds for infinitely many $n$ then such a decomposition does not exist. To illustrate the problem we give some examples.

\begin{example}
Let $a_{2n-1}=a_{2n}=\frac{1}{2^n}$ for all $n$. Then the sequence has a decomposition into two interval-filling subsequences containing the odd and the even terms respectively. We have $q_{2n-1}=\frac{1}{3}$ and $q_{2n}=\frac{1}{2}$ for each $n$.
\end{example}

\begin{example}
Let $a_{n}=(\frac{\sqrt{2}}{2})^n$ for all $n$. Then the sequence has a decomposition into two interval-filling subsequences containing the odd and the even terms respectively (both are geometric with the ratio $\frac{1}{2}$). We have $q_{n}=\frac{1}{\sqrt{2}+1}$ for every $n$.
\end{example}

Now we consider the sufficient condition for having the alternating decomposition into two fast convergent subsequences.
\begin{theorem}
Let $(a_n)$ be an interval-filling sequence. If $a_n>r_{n+1}$ for all $n$ then there exists a decomposition of $(a_n)$ into two fast convergent sequences.
\begin{proof}
The following inequalities hold
$$a_n>r_{n+1}>r_{n+1}^{(2)}$$
Hence the sequences $(a_{2n-1})$ and $(a_{2n})$ are fast convergent.
\end{proof}
\end{theorem}

In the studies of decompositions into two fast convergent sequences the major characteristic is the value of the ratio $p_n:=\frac{a_n}{r_{n+1}}$. Clearly, $p_n>q_n$. Note that if $p_n\leq 1$ holds for all $n$, then we get the definition of a locker, which is a stronger notion to the interval-filling sequence and was described in \cite{DJK} and \cite{DK}. If $p_n>1$ for all $n$, then we know that $(a_n)$ can be decomposed into two fast-convergent sequences. On the other hand, if $p_n\leq\frac{1}{2}$ for all large enough $n$ then we get that such decomposition does not appear.

We finish the consideration with the example which does not satisfy the sufficient condition, but for which the decomposition exists.

\begin{example}
Let $(c_n)$ be any decreasing sequence with elements from the interval $(\frac{1}{2},1)$. Let $a_{2n-1}=a_{2n}=\frac{c_n}{2^n}$.
Note that $(a_n)$ is interval-filling. Indeed, for $n \in \N$ we have $a_{2n-1}=a_{2n} < r_{2n-1}$ and 
$$a_{2n}=\frac{c_n}{2^n}<\frac{1}{2^n}=\sum_{k=1}^{\infty}\frac{1}{2^{n+k}}=\frac{1}{2}\sum_{k=1}^{\infty}\frac{1}{2^{n+k}} +\frac{1}{2}\sum_{k=1}^{\infty}\frac{1}{2^{n+k}}<\sum_{k=1}^{\infty}\frac{c_{n+k}}{2^{n+k}} +\sum_{k=1}^{\infty}\frac{c_{n+k}}{2^{n+k}}=r_{2n}.$$
Moreover, it can be decomposed into two fast convergent sequences $(a_{2n-1})$ and $(a_{2n})$ of equal terms, since
$$a_{2n}=\frac{c_n}{2^n}=\sum_{k=1}^{\infty}\frac{c_n}{2^{n+k}}>\sum_{k=1}^{\infty}\frac{c_{n+k}}{2^{n+k}}=r_{2n}^{(2)}$$
On the other hand, for the value of the ratios we get
$$p_{2n-1}=\frac{a_{2n-1}}{r_{2n}}=\frac{c_{n}}{2^{n}}\cdot\frac{1}{2\sum_{k=1}^{\infty}\frac{c_{n+k}}{2^{n+k}}}<\frac{c_n}{2^n}\cdot\frac{1}{2\sum_{k=1}^{\infty}\frac{c_{n}}{2^{n+k}}}=\frac{1}{2}<1.$$
\end{example}

Now, we give the example of an interval-filling sequence such that none of its decompositions contains two interval-filling or two fast convergent sequences.

\begin{example}
We will find a sequence which cannot be decomposed neither into two interval filling subsequences, nor into two fast convergent subsequences. Let $(a_n)$ be defined as follows: $a_{3n-2}=a_{3n-1}=q^{2n-1}$, $a_{3n}=q^{2n}$ for each $n$, that is $(a_n)$ is the multigeometric sequence $(a_n)=\frac{1}{q}(1,1,q;q^2)$. We will now find a proper $q \in (0,1)$. First, we need to break the necessary condition for decomposition into two interval-filling sequences given in Theorem \ref{nadwaprzedzialy}. Several inequalities should be satisfied:
\begin{itemize}
\item $a_{3n-2}+a_{3n-1}> r_{3n-1}\Leftrightarrow 2q^{2n-1}> \frac{q^{2n}+2q^{2n+1}}{1-q^2}\Leftrightarrow 4q^2+q-2< 0\Leftrightarrow q< \frac{1}{8}(\sqrt{33}-1)\approx 0,593$
\item $a_{3n-1}+a_{3n}> r_{3n}\Leftrightarrow q^{2n-1}+q^{2n}> \frac{2q^{2n+1}+q^{2n+2}}{1-q^2}\Leftrightarrow 2q^3+3q^2-q-1< 0\Leftrightarrow q< \frac{1}{2}(\sqrt{5}-1)\approx 0,618$
\item $a_{3n}+a_{3n+1}> r_{3n+1}\Leftrightarrow q^{2n}+q^{2n+1}>q^{2n+1}+ \frac{q^{2n+2}+2q^{2n+3}}{1-q^2}\Leftrightarrow 2q^3+2q^2-1< 0\Leftrightarrow \\q< \frac{1}{6}(\sqrt[3]{46-6\sqrt{57}}+\sqrt[3]{46+6\sqrt{57}}-2)\approx 0,565$
\end{itemize}
Hence for $q$ satisfying the above three inequalities (the third of them determines the upper boundary) we know that $(a_n)$ has no decomposition into two interval-filling sequences.
\\Suppose that $(a_n)$ has a decomposition into $(y_n)$ and $(z_n)$, where both of them are fast convergent. Since some of the terms in $(a_n)$ repeats two times, we need to divide it between our two sequences. Hence for every $n$ exactly one of the terms $a_{3n-2}$ and $a_{3n-1}$ belongs to $(y_n)$ and the other one to $(z_n)$. Thus, $(y_n)\supset (q^{2n-1})$ and $(z_n)\supset (q^{2n-1})$. Note that $q^2$ need to belong to one of them, let us assume that $q^2\in(y_n)$.
 \\Now, let us consider for which $q$ the sequence $(y_n)$ is not fast convergent.
$$y_1=q\leq q^2+\sum_{k=1}^{\infty}q^{2k+1}\leq r_1^y\Leftrightarrow q\leq q^2+\frac{q^3}{1-q^2}\Leftrightarrow q\geq q_0\approx 0.555.$$
Hence we have obtained the interval to choose the proper $q$, in particular $q=0,56$. Then, the sequence $(a_n)$ has no decomposition into two fast convergent sequences.

Note that the example shows even more. Because of its self similar multigeometric structure, we can not decompose $(a_n)$ into two sequences both of which satisfy either slow or fast convergence condition for large enough indexes.
\end{example}

\begin{problem}
Characterize the interval-filling sequences $(a_n)$ for which the decomposition into two or more interval-filling sequences or into two fast convergent sequences is possible in terms of sequences $(p_n)$ and $(q_n)$ respectively.
\end{problem}


\end{document}